\documentclass{article}
\usepackage[T1]{fontenc}
\usepackage[utf8]{inputenc}
\usepackage[english]{babel}
\usepackage[margin=3cm]{geometry}
\usepackage{amsmath, amssymb}
\usepackage{amsfonts}
\usepackage{dsfont}
\usepackage{float}
\usepackage{graphicx}
\usepackage{wrapfig}
\usepackage{mathtools}
\usepackage{bbm}
\usepackage{amsthm}
\usepackage{ifthen}
\usepackage{graphicx}
\usepackage{hyperref}
\usepackage[ruled,vlined]{algorithm2e}
\usepackage{xcolor}
\usepackage{mathtools}
\usepackage{empheq}
\usepackage{mathrsfs}
\usepackage{multicol}

\newtheorem{theorem}{Theorem}[section]
\newtheorem{prop}[theorem]{Proposition}
\newtheorem{definition}[theorem]{Definition}
\newtheorem{remark}{Remark}[section]
\newtheorem{lemma}[theorem]{Lemma}

\newtheorem{assumptions}{Assumptions}[section]

\newcommand{\ca}[1]{\mathcal{#1}}
\newcommand{\bb}[1]{\mathbb{#1}}
\newcommand{\p}{\mathbb{P}}

\newcommand{\comillas}[1]{``#1''}
\newcommand{\set}[1]{\left\{#1\right\}}
\newcommand{\parent}[1]{\left(#1\right)}

\newcommand{\borel}{\ca{B}(\bb{R}^d)}
\newcommand{\Rd}{\bb{R}^d}
\newcommand{\R}{\bb{R}}

\newcommand{\ind}[1]{\mathbbm{1}_{#1}}
\newcommand{\esp}[1]{\bb{E}\barras{#1}}

\newcommand{\gb}[1]{\overline{\widehat{#1}}}
\newcommand{\gbt}[2]{\overline{\widehat{#1}}_{t_{#2}}}
\newcommand{\barras}[1]{\left| #1 \right|}
\newcommand{\integral}{\int_{t_i}^{t_{i+1}}}
\newcommand{\ug}[1]{\hat{u}_{#1}}
\newcommand{\vg}[1]{\widehat{\ca{V}}_{#1}}
\newcommand{\norm}[1]{\left\lVert#1\right\rVert}
\newcommand{\xscheme}[1]{X_{t_{#1}}^{\pi}}
\newcommand{\prom}[1]{\langle #1 \rangle}

\newcommand{\bra}[2]{\langle #1, #2 \rangle}

\begin{document}

\title{The Kolmogorov Infinite Dimensional Equation in a Hilbert space Via Deep Learning Methods}
\author{Javier Castro\footnote{address: Departamento de Ingenier\'{\i}a Matem\'atica, Universidad de Chile, Casilla
170 Correo 3, Santiago, Chile. email: jcastro@dim.uchile.cl}
\footnote{J.C.'s work was partially funded by Fondecyt no. 1191412 and CMM Projects ``Apoyo a Centros de Excelencia'' ACE210010 and Fondo Basal FB210005.}}

\numberwithin{equation}{section}

\maketitle

\begin{abstract}
We consider the nonlinear Kolmogorov equation posed in a Hilbert space $H$, not necessarily of finite dimension. This model was recently studied by Cox et al. \cite{CJL} in the framework of weak convergence rates of stochastic wave models. Here, we propose a complementary approach by providing an infinite-dimensional Deep Learning method to approximate suitable solutions of this model. Based in the work by Hure, Pham and Warin \cite{DBS} concerning the finite dimensional case, and our previous work \cite{yo} dealing with L\'evy based processes, we generalize an Euler scheme and consistency results for the Forward Backward Stochastic Differential Equations to the infinite dimensional Hilbert valued case. Since our framework is general, we require the recently developed DeepOnets neural networks \cite{chen95,error-estimates-DOnets} to describe in detail the approximation procedure. Also, the framework developed by Fuhrman and Tessitore \cite{fuhrman} to fully describe the stochastic approximations will be adapted to our setting.

%In this paper we consider the numerical approximation of nonlocal integro differential parabolic equations via neural networks. These equations appear in many recent applications, including finance, biology and others, and have been recently studied in great generality starting from the work of Caffarelli and Silvestre \cite{CS}. Based in the work by Hure, Pham and Warin \cite{DBS}, we generalize their Euler scheme and consistency result for Backward Forward Stochastic Differential Equations to the nonlocal case. We rely on L\`evy processes and a new neural network approximation of the nonlocal part to overcome the lack of a suitable good approximation of the nonlocal part of the solution.

\end{abstract}
\tableofcontents

\section{Introduction}
Let $H,V $ be separable Hilbert spaces with inner products $\langle\cdot,\cdot\rangle_H$ and $\langle\cdot,\cdot\rangle_V$, and $T>0$. We consider the infinite dimensional Kolmogorov model
\begin{equation}
	\left \{
	\begin{aligned}
		\partial_t u(t,x) + \ca{L}[u](t,x) + \psi \big( t,x,u(t,x),B^* (t,x)\nabla u(t,x) \big)&=0, && (t,x)\in[0,T]\times H ,\\
		u(T,x) &= \phi(x), && x\in H.
	\end{aligned}
	\right.
	\label{eq:pde}
\end{equation}
Here $u\colon[0,T]\times H\to \R$ is the unknown of the problem, $B^*$ is the formal adjoint of a suitable mapping $B$, $\phi\colon H\to\R$ is a terminal condition and $\psi$ represents the non-linear character of the problem. $\nabla$ represents the Fréchet derivative with respect to the spatial variable $x\in H$. Finally, the operator $\ca{L}$ is defined for $f\in C^{0,2}([0,T]\times H)$ and $(t,x)\in[0,T]\times H$. The precise details on these terms are fixed below in Assumptions \ref{assumptions}.

\medskip

In the case where $H=\R^d$ equation \eqref{eq:pde} can be recast as a nonlinear parabolic model, generalizing the classical Heat equation. The mathematical theory in this case is well-known, see e.g. \cite[Section $2.3$]{Evans}. Of great importance to the present work is the well known relation between probabilities and parabolic models, A. N. Kolmogorov was the first (of many) to notice these relations in his foundational work \cite{kolmogorov}, the resulting theory  allows to prove existence, uniqueness and properties of solutions to parabolic models, known as Kolmogorov equations, by means of probabilistic ideas. These models, also known as diffusion equations, has many applications in Finance and other areas such as physics, biology, chemistry and economics. The success in applications came from the fact that these equations are describing the general phenomena of particles interacting under the influence of random forces (see e.g. \cite{diffusion-crank}). 

\medskip

Following Kolmogorov ideas, we also consider a decoupled system of stochastic partial differential equations (SPDEs) for $(X_t,Y_t,Z_t)_{t\in [0,T]}$
\begin{align}
	X_{t}&=x + \int_0^t (AX_s + F(s,X_s))ds+\int_0^t B(s,X_{s}) dW_s,
	\label{eq:fpsde}\\
	Y_{t}&=\phi(X_T)+\int_{t}^T \psi(s,X_s,Y_s,Z_s)ds-\int_{t}^T \prom{Z_s,\cdot}_{0} dW_s,
	\label{eq:bpsde}
\end{align}
where $\prom{\cdot,\cdot}_0$ is a suitable $\mathcal L$ based inner product to be defined below.  Forward Backward SPDEs (FBSPDEs) such as system \eqref{eq:fpsde}-\eqref{eq:bpsde} were first studied by Pardoux and Peng in the finite dimensional case \cite{pardoux90}, whereas Barles, Buckdahn and Pardoux \cite{pardoux} generalized it to the case where also a non continuous process is considered. For the stochastic equation posed on infinite dimensional spaces, we refer to the book  \cite{daprato} and articles \cite{albeverio,fuhrman}. 

\medskip

However, in the infinite dimensional case, \eqref{eq:pde} becomes a highly complicated model that requires sophisticated treatment and generalizations for the classical existence and regularity theories. These equations were first investigated by Yu. Daleckij \cite{yu} and L. Gross \cite{gr}. As you may see in Section \ref{sec:Functional Numerical Scheme}, to apply our method we need a {\bf strong solution} of \eqref{eq:pde} because we make use if the It\^o lemma. The existence of said solution can be seen as a strong assumption in our model, nevertheless, there are results on the existence and uniqueness of weaker types of solutions. In \cite{fuhrman}, \textit{mild solutions} of \eqref{eq:pde} are introduced. A function $u\colon[0,T]\times H\to\R$ is called a {\bf mild solution} to \eqref{eq:pde} if it satisfies $u\in C^{0,1}([0,T]\times H)$, there exists $C>0$ and $p\in\bb{N}$ such that $|\prom{\nabla u(t,x),h}_H|\le C\norm{h}_H(1+\norm{x}_H^p)$ for all $t\in[0,T]$ and $x,h\in H$ and the following weaker formulation of \eqref{eq:pde} is satisfied
\begin{align*}
	u(t,x) = -\int_t^T \bb{E}\parent{\psi(s,X_{s}^{t,x},u(s,X_{s}^{t,x}),G(s,X_{s}^{t,x})^{*}\nabla u(s,X_{s}^{t,x}))}ds + \bb{E}\phi(X_{s}^{t,x}).
\end{align*}
Where $(X_{s}^{t,x})_{s\in[t,T]}$ is the solution to the forward stochastic equation \eqref{eq:fpsde} starting with $X_{t}^{t,x}=x$. In \cite{fuhrman} the authors prove that there exists a unique mild solution to \eqref{eq:pde} which is related to the stochastic equations through $u(t,x)=Y_t^{t,x}$, where $Y^{t,x}$ is part of the solution to the backward equation in $[t,x]$ starting with $X_t^{t,x}=x$.

\medskip

The mathematics presented here is strongly inspired by the article \cite{DBS} written by Hure, Pham and Warin, where they rely on the stochastic representation of \eqref{eq:pde} (with $H=\Rd$) and the use of neural networks to approximate a solution of the PDE and its spatial gradient. Due to the importance of this work to the present article, we provide a detailed enough description of the scheme presented in \cite{DBS} and certain ideas of generalization; Consider a partition $\pi$ of $[0,T]$. By taking advantage of the relations $Y_t=u(t,X_t)$ and $Z_t = \sigma^{T}(t,X_t)\nabla u(t,X_t)$ showed in \cite{pardoux90} and the It\^o formula, mentioned authors proposed a pair of neural networks $\ca{U}_t(\cdot;\theta)$ and $\ca{Z}_t(\cdot;\theta)$ for every $t\in\pi$ such that
\begin{align*}
	\ca{U}_t(X^{\pi}_t;\theta)\approx Y_t\ \text{and}\ \ca{Z}_t(X^{\pi}_t;\theta)\approx Z_t,\ t\in\pi.
\end{align*}
Where $X^{\pi}$ is a suitable Euler approximation of the diffusion $X$ and $\theta$ represents the neural network parameters (see \cite[Section $3$]{DBS} for notation and note that matrix $\sigma$ in \cite{DBS} is a particular case of $B$). Recall that the work by Hure et al. is posed in an finite dimensional framework. Then, by imposing that the neural network representation satisfies the Ito formula with a cost incurred by the approximation, an iterative backward induction is produced such that at each time step a loss function representing the cost is minimized. This process generates optimal neural networks for every time step $t\in\pi$. The backward form of the algorithm emerges from the knowledge of the solution at the final time, also known as terminal condition. It is important to mention that Hure et al. extend this approach to treat variational inequalities. Still in finite dimension, our previous work \cite{yo} added a non-local term to the considered PDE. This modification introduces complications such as the need of a general diffusion which admits discontinuities. This type of processes are known in the literature as Lévy processes and are suitable to obtain the desire representation as in the local case (see \cite{pardoux}). Examples of non-local terms includes integrals with respect to a Levy measure $\lambda$ (see \cite{yo} for details), but only finite Levy measures are taking under consideration in the commented article, this restriction leaves out interesting operators such as fractional laplacian. Other important complication presented in \cite{yo} is the need of a third neural network to approximate the non-local term. A different approach to treat the non-locality is considered by Lukas Gonon and Christoph Schwab in \cite{lukas-gonon1, lukas-gonon2}, they prove that NNs of a particular form are able to approximate expectation of a certain type of functions defined on the space of stochastic processes with jumps, which can express particular types of PDEs solutions. In their proof, L. Gonon and C. Schwab provide dimension-explicit bounds evidencing that their scheme is free from the \emph{curse of dimensionality} (mentioned below). These articles are a generalization of the method presented previously in \cite{intro2, jentzen-nn-thoery}, for more details see the references there in.

\medskip

In a recent work by Cox, Jentzen and Lindner \cite{CJL}, the authors investigate a temporal discretization of the stochastic wave equation which is a special case of \eqref{eq:fpsde}. Furthermore, they establish weak convergence rates for the said discretization by employing the recent mild Ito formula discussed in \cite{mild-ito-jentzen}. The latter work deals with a weaker notion of stochastic processes which they define as \emph{mild stochastic processes}. These objects arise naturally from considering weaker solutions of stochastic partial differential equations (SPDE), and consistently, these solutions are known in the literature as \emph{mild solutions}, see \cite[Proposition $7.1$]{daprato} or Definition \ref{def:strong-mild} for a view of these concepts. For this type of solutions, the authors of \cite{mild-ito-jentzen} introduced a version of It\^o formula which suggests the existence of an infinite dimensional version of the Kolmogorov equation, and becomes one of our main sources of inspiration to describe a Hilbert generalization of \cite{DBS}. Recall that SPDEs have, by definition, a Hilbert or Banach space framework, and a conveniently mild It\^o formula is even defined for SPDEs posed on very general Banach spaces, see \cite{mild-ito-banach}.    

\medskip

In recent developments, finite dimensional Deep Learning (DL) has proven itself to be an efficient tool to solve nonlinear problems such as the approximation of PDEs solutions (see \cite{state-art-dl}). In particular, in high dimensions $d\gg 1$, typical methods such as finite difference or finite elements suffer from the fact that the complexity of the problem grows exponentially on $d$, problem known in the literature as \emph{curse of dimensionality}. Without being exhaustive, we present some of the current developments in this direction. First of all, Monte Carlo algorithms are an important and widely used approach to the resolution of the dimension problem. This can be done by means of the classical Feynman-Kac representation that allows us to write the solution of a linear PDE as an expected value, and then approximate the high dimensional integrals with an average over simulations of random variables. On the other hand, Multilevel Picard method (MLP) is another approach and consists on interpreting the stochastic representation of the solution to a semilinear
parabolic (or elliptic) PDE as a fixed point equation. Then, by using Picard iterations together with Monte Carlo methods for the computation of integrals, one is able to approximate the solution to the PDE, see \cite{intro1, intro2} for fundamental advances in this direction. As another option, the so-called Deep Galerkin method (DGM) is another DL approach used to solve quasilinear parabolic PDEs of the form $\ca{L}(u)=0$ plus boundary and initial conditions. The cost function in this framework is defined in an intuitive way, it consists of the differences between the approximated solution $\hat u$ evaluated at the initial time and spatial boundary, with the true initial and boundary conditions plus $ \ca{L}(\hat u)$. These quantities are captured by an $L^2$-type norm, which in high dimensions is minimized using Stochastic Gradient Descent (SGD) method. See \cite{intro3} for the development of the DGM and \cite{intro4} for an application. The article \cite{EHJ17} by E, Han and Jentzen, is considered one of the first attempts to solve this issue by means of Deep Learning (DL) techniques. In said paper, the authors proposed an algorithm for solving parabolic PDEs by reformulating the problem as a stochastic control problem. This connection also come from the Feynman-Kac representation, proving once more that stochastic representations are a key tool in the area. More recent developments in this area can be found in Han-Jentzen-E \cite{HJE} and Beck-E-Jentzen \cite{intro5}.

\medskip

Naturally, one has to distinguish between the SPDE and the infinite dimensional PDE and work each one separately. Both are highly complicated equations to solve numerically, or even to propose a proper discretization method which may or may not be implementable. Here we are only interested in working in the PDE side of the problem by assuming a relatively good numerical scheme for the stochastic side of it. The scheme presented here is, indeed, numerically implementable. Nevertheless, in this article we chose not to present numerical results, but instead to give a proof of the consistency of this algorithm. Our proof is the generalization of the one given in \cite{DBS} to the infinite dimensional case. 

\medskip

The problem to generalize neural networks to an infinite dimensional framework has been investigated in dynamical systems and PDEs. In our case, following \cite{DBS,yo} given the partition $\pi=\set{t}_{t\in\pi}$ of $[0,T]$, we want to approximate the solution $u(t,\cdot)$ to \eqref{eq:pde} and a fixed function of its gradient $\nabla u(t,\cdot)$ for $t\in\pi$, which in general are nonlinear operators from $H$ to some other separable real Hilbert space $(W,\prom{\cdot,\cdot}_W,\norm{\cdot}_W)$. Thus, we need a general Deep Learning framework which considers the approximation of operators $F\colon H\to W$ by a neural network $F^{\theta}\colon H\to W$, where $\theta$ is a finite dimensional parameter. Sandberg \cite{S91} defined a set of infinite dimensional mappings parameterized by finite dimensional parameters, providing a universal approximation theorem for those mappings. Other important article in the development of infinite dimensional neural networks and an key reference for the theory presented here, is \cite{chen95} by Chen and Chen. They deal with the  approximation of mappings defined on a compact subset of $C(K)$ with values in $\R$ and $C(K)$, where $K$ is a compact subset of a finite dimensional space. A key lemma (\cite[Lemma $7$]{chen95}) presented in there says that, for a compact set $V$ in $C(K)$, one can consider a transformation $T$ and define $T(V) = \set{Tu\colon u\in V}$ such that every function in $V$ is close to its transformation. One can compare these ideas with the approximation of measurable functions with simple functions in integral-type distance and continuous functions with polynomials in uniform norm. The transformed set is constituted by, in some sense, simpler functions that can be easily described by finite dimensional neural networks which allows them to create a proper architecture. Lemma \ref{lemma:hilbert_aris} is the counterpart of \cite[Lemma $7$]{chen95} for a compact set $V$ in a Hilbert space. Here, the considered transformation is the projection onto a finite set of an orthonormal basis. Chen and Chen also demonstrate that their architectures approximate any continuous mapping in uniform norm. More recently Lu, Jin and Karniadakis, based on \cite{chen95}, introduced an architecture called {\bf DeepONets} \cite{LJK19}, which are mappings between spaces of continuous functions. DeepONets rely on representing the input function and its evaluation on a fixed finite set of points. Then, via an activation function, one takes the finite dimensional information to an element of the set of continuous functions.\\

It is common in machine learning and, more generally, in some statistics frameworks, to consider mean square error due to its convexity properties. Here this framework emerges naturally because we make use of stochastic processes, which will be essentially square integrable random variables. The quantity used to measure the error incurred in our scheme will depend on how good our architectures are able to approximate elements of $L^2(H,\mu;W)$. Here, $\mu$ is the law of an $H$-valued random variable $X$ (this random variable will be related to a stochastic process). Then, it is natural to consider the $L^2$-distance or mean square error
\begin{align*}
	\bb{E}\norm{F(X)-F^{\theta}(X)}_{W}^2 = \int_{H} \norm{F(x)-F^{\theta}(x)}_{W}^2 \mu(dx),
\end{align*}
where $F$ is some mapping and $F^{\theta}$ the proposed architecture. The purpose of this paper is to describe solutions of the infinite dimensional Kolmogorov equation using recent Deep Learning techniques. More precisely, we will find infinite dimensional neural networks of type Deep-H-Onets (to be defined in this article) that approximate suitable solutions of \eqref{eq:pde}. This is done in our main result, Theorem \ref{MT1}. 

\medskip

\noindent
{\bf Organization of this paper.} This paper is organized as follows. In Section \ref{sec:pre} we provide notation and required background. In Section \ref{sec:the-forward-backward-stochastic-system}, a detailed description of the stochastic equations is presented. In Section \ref{sec:Functional Numerical Scheme}, we introduce the analyzed numerical scheme. In section \ref{sec:Universal Approximation Theorems and Deep-H-Onets}, we define neural networks and infinite dimensional neural networks and prove classical universal approximation results. In the last section, Section \ref{sec:Main Result}, we state and prove the main result of our work.

\medskip

\noindent
{\bf Acknowledgments.} We want to thank professor Aris Daniilidis for helping us with some deep functional analysis topics and useful discussions, see Lemma \ref{lemma:hilbert_aris}. We would also like to thank professor Julio Backhoff for his insights on the existence of a strong solution to \eqref{eq:pde}.

\section{Preliminaries}\label{sec:pre}

\subsection{Notation}\label{sec:notation} We cannot continue without introducing some notation needed to state our main result. 

\medskip

{\bf Finite dimension}. For any $m\in\bb{N}$, $\R^m$ represents the finite dimensional Euclidean space with elements $x=(x_1,...,x_m)$ endowed with the usual norm $\norm{x}_{\R^m}^2=\sum_{i=1}^m |x_i|^2$. We will simply write $\norm{x}$ when no confusion can arise. Note that for scalars $a\in\R$ we also denote its norm as $|a| = \sqrt{a^2}$. For $x,y\in\R^m$ their scalar product is denoted as $x\cdot y =\sum_{i=1}^m x_i y_i$. Finally, along this paper we will use several times that for $x_1,...,x_k\in \R$, the following bound holds,
\begin{equation}\label{eq:square-bound}
	(x_1+\cdots+x_k)^2\le k(x_1^2+\cdots+x_k^2).
\end{equation}

{\bf Banach spaces}. Consider now two real Banach spaces $E,F$. Given a subset $A\subset E$ we denote as $\prom{A}$ the set containing all the finite linear combination of elements in $A$. For a separable real Hilbert space $(H,\prom{\cdot,\cdot}_H, \norm{\cdot}_H)$, we denote by $(e_i)_{i\in\bb{N}}$ a countable orthonormal basis. We denote by $C^m(E;F)$ the set of all $m$ times continuously differentiable functions from $E$ to $F$ and $C^m(E)$ when $F=\R$. $L(E,F)$ denotes the space of continuous linear functions from $E$ to $F$ endowed with the usual operator norm, and by $L_2(H,F)$ we mean the set of Hilbert-Schmidt operators $A\in L(H,F)$ such that $\norm{A}_{L_2}^2=\sum_{k=1}^{\infty}\norm{Ae_k}^2_F<\infty,$ endowed with the corresponding norm. 

\medskip

{\bf Measures}. We also denote by $\ca{B}(E)$ the Borel $\sigma$-algebra on $E$. For a general measure space $(E,\ca{H},\nu)$ and $p\ge 1$, $L^p(E,\ca{H},\nu;F)$ represents the standard Lebesgue space of all $p$-integrable functions from $E$ to $F$, with its Borel $\sigma$-algebra, and endowed with the norm
\begin{align*}
	\norm{f}^p_{L^p(E,\ca{H},\nu;F)} = \int_{E} \norm{f(x)}_F^p\nu (dx).
\end{align*}
We write $L^p(E,\ca{H},\nu)$ when $F=\R$ and $L^p(E,\nu)$ when $F=\R$ and $\ca{H}$ is the Borel $\sigma$-algebra $\ca{B}(E)$. See the \comillas{Appendix A} section of \cite{wei} for a definition of the above Bochner integral and its properties. We also write
\[
\int_E f(s)ds = \begin{pmatrix}\int_E f_1(s) ds\\ \vdots \\ \int_E f_m(s) ds\end{pmatrix},
\]
whenever $f:E\to\R^m$ with $f=(f_1,...,f_m)$. 

\medskip

{\bf Stochastic processes}. We refer to \cite{daprato} for a detailed development of Stochastic Calculus in infinite dimensions. Here we will need the following definitions. 

\medskip

Let $(\Omega,\ca{F},\p)$ be a complete probability space. Given a $E$-valued random variable $X:\Omega\to E$, we write $\bb{E}X = \bb{E}(X)$. We denote by $\sigma(X)$ the $\sigma$-algebra generated by $X$ and by $\ca{P}_s$ the predictable $\sigma$-algebra of $[0,s]\times\Omega$. Let us denote by $\mathscr{S}^2 = \mathscr{S}^2_T(E)$ the space of $E$-valued predictable processes $(X_t)_{t\in[0,T]}$ endowed with the norm $\norm{X}_{\mathscr{S}^2}=\bb{E}\left(\underset{t\in [0,T]}{\sup} \norm{X_t}^2_E\right)$. We denote $\mathscr{M}_T^2(E)\subset\mathscr{S}^2_T(E)$ the space of $E$-valued continuous, square integrable martingales $(M_t)_{t\in[0,T]}$ such that $M_0 = 0$ endowed with the norm $\norm{M}_{\mathscr{M}^2} = \norm{M}_{\mathscr{S}^2}$. Note that if $X\in L^2(\Omega,\ca{F},\p;E)$, then $M_t=\bb{E}(X|\ca{F}_t)$ defines a martingale in $\mathscr{M}^2_T(E)$. We also have that if $M$ is a continuous martingale, then Doob's inequality holds,
\begin{align*}
	\bb{E}\left( \underset{t\in[0,T]}{\sup}\norm{M_t}_E^2 \right) \le 4\underset{t\in[0,T]}{\sup}\left(\bb{E}\norm{M_t}_E^2\right).
\end{align*}
If no confusion arises, we will drop the parentheses $(\cdot)$ in each $\bb{E}$.

\subsection{Stochastic Calculus on Hilbert Spaces}
In this Subsection we gather some necessary results needed in the proof of the main result. We first state a series of properties on Stochastic Calculus posed in Hilbert Spaces. For a detailed view, see \cite[Section $4$]{daprato}.

\medskip

 Consider the real separable Hilbert space $(V,\prom{\cdot,\cdot}_V,\norm{\cdot}_V)$ with an orthonormal basis $(f_k)_{k\in\bb{N}}$ and $Q\in L(V)$ be a trace class nonnegative operator, which means $\sum_{k=1}^{\infty}\prom{Qf_k,f_k}<\infty$. Define now 
 \[
 V_0=Q^{1/2}V=\set{Q^{1/2}v\ \big|\ v\in V},
 \] 
which is another Hilbert space endowed with $\prom{u_0,v_0}_0=\prom{Q^{-1/2}u_0,Q^{-1/2}v_0}_V$ and the corresponding norm $\norm{\cdot}_0$. Operator $Q$ will appear in the definition of the operator $\mathcal L$ in Assumptions \ref{assumptions}.

\medskip

For a Hilbert space $K$ let $L_{2}(V_0,K)$ be the set of Hilbert-Schmidt operators defined on $V_0$ and taking values in $K$.
\begin{remark}\label{remark:schmidt-dual}
	Note that $L(V,K)\hookrightarrow L_2(V_0,K)$. Also, observe that if $K=\R$, then for every $v\in L(V,\R)=V^{*}$ (up to isomorphism),
	\begin{align*}
		\norm{v}_{L_2(V,\R)}^2
		= \sum_{j=1}^{\infty} |\prom{v,f_j}|^2 = \norm{v}_V^2.
	\end{align*}
	Therefore in this particular case $L(V,\R)=L_2(V,\R)$. 
\end{remark}
Recall $Q$ as introduced before. A $V$-valued process $(W_t)_{t\ge 0}$ is called a $Q$-Wiener process if
\begin{enumerate}
	\item[$(i)$] $W_0=0$,
	\item[$(ii)$] $W$ has continuous trajectories and independent increments and,
	\item[$(iii)$] The law $\mathscr{L}(W_t - W_s)=\mathscr{N}(0,(t-s)Q)$ for $t\ge s\ge 0$, i.e. the Gaussian measure with mean $0$ and covariance operator $(t-s)Q$.
\end{enumerate}  
We shall assume

\begin{assumptions}\label{Ass2p1}
There exists a bounded sequence of nonnegative real numbers $(\lambda_k)_{k\in\bb{N}}$ such that $Q f_k = \lambda_k f_k$ for $k\in\bb{N}$. 
\end{assumptions}
Due to $Q$ been trace class, one can prove that $\text{Tr}(Q)=\sum_{k=1}^{\infty} \lambda_k < \infty$, result known as Lidskii's theorem. We provide an example of trace class operator. Consider the usual Hilbert space $H$ (could be any Hilbert space) and $x,y\in H$, define the bounded linear operator $T_{x,y}\in L(H)$ such that $T_{x,y}z=\prom{z,y}_H x$ for any $z\in H$. Then $\text{Tr}(T_{x,y}) = \prom{x,y}_H$. Furthermore, any bounded linear operator with finite-dimensional rank is trace class.

\medskip

For a $V $-valued $Q$-Wiener process $(W_t)_{t\in[0,T]}$ we have the representation \cite{daprato}
\begin{align}\label{eq:w-representation}
	W_t = \sum_{k=1}^{\infty} \sqrt{\lambda_k}\beta^k_t f_k\ \ \text{with}\ \ \beta^k_t = \frac{1}{\sqrt{\lambda_k}}\bra{W_t}{f_k}_{V},
\end{align}
where the series converges in $L^2(\Omega,\ca{F},\p;V)$ and $(\beta^j)_{j\in\bb{N}}$ is a sequence of independent real valued Brownian motions on $(\Omega,\ca{F},\p)$. For $n\in\bb{N}$ consider
\begin{align}\label{eq:w-n}
	W^n_t = \sum_{k=1}^n \sqrt{\lambda_k}\beta^k_t f_k,\ t\in [0,T].
\end{align}
\begin{definition}\label{def:integrable-processes}
	For a Hilbert space $K$ (usually $\R$ or $H$), we define the set $\mathscr{N}_W^2(0,T;L_2(V_0,K))$ of $L_2(V_0,K)$-valued predictable processes $\Phi\colon[0,T]\times\Omega\to L_2(V_0,K)$ such that
	\begin{align*}
		\norm{\Phi}^2_{\mathscr{N}_W^2(0,T;L_2(V_0,K))}=\bb{E}\int_0^T \| \Phi_s \|^2_0 ds < \infty,
	\end{align*}
	endowed with the corresponding norm, i.e. $\norm{\cdot}_{\mathscr{N}_W^2(0,T;L_2(V_0,K))}$ which we also denote as $\norm{\cdot}_{\mathscr{N}_W^2}$ when no confusion arises.
\end{definition}
Such processes are suitable for integrate with respect to $(W_t)_{t\in[0,T]}$ obtaining another stochastic process
\begin{align}\label{eq:stochastic-intregral}
	\int_0^t \Phi_s dW_s,\ t\in[0,T],
\end{align}
which is a continuous square integrable martingale. See \cite[Section $4.3$]{daprato} for properties of this integral.
\subsection{Some useful lemmas}
In this section we have compiled some basic but essential facts that will be used in the proof for introductory results to state main Theorem \ref{MT1}. Of particular importance is the \textit{Martingale Representation Theorem} \ref{lemma:mg-representation} which allows us to find a solution for the backward stochastic equation.
\begin{lemma}
	 The integral \eqref{eq:stochastic-intregral} can be approximated as follows: for $n\in\bb{N}$ consider the Wiener process $(W^n_t)_{t\in [0,T]}$ in \eqref{eq:w-n}, then
	\begin{align*}
		\bb{E}\parent{\underset{t\in [0,T]}{\sup}\norm{\int_0^t \Phi_s d W_s - \int_0^t \Phi_s d W^n_s}^2}\to 0\quad\text{as}\quad N\to\infty,
	\end{align*}
	for any $(\Phi_s)_{s\in [0,T]}\in\mathscr{N}_W([0,T];L_2(V_0,K))$.
\end{lemma}
\begin{lemma}\label{lemma:n-brownian-integral}
	Let $n\in\bb{N}$ and $(\Phi_s)_{s\in[0,T]}\in \mathscr{N}_W (0,T;L_2(V_0,H))$, then the following holds,
	\begin{align*}
		\int_{0}^t \Phi(s)dW_s^n = \sum_{j=1}^n \int_0^t\Phi(s)(Q^{1/2}f_j)d\beta^j
		_s.
	\end{align*}
	Where $W^n$ is given by \eqref{eq:w-n}. 
\end{lemma}
\begin{proof}[\textbf{Proof}:]
	First, note that we have $n$ integrals of $H$-valued processes with respect to real valued standard Brownian Motions (the associated covariance operator in this case is just $1$). In our case, the space that gives sense to these integrals is $\mathscr{N}_W(0,T;L_2(\R,H))$. It is straightforward that $L_2(\R,H)=H$. We proceed by proving the property for elementary processes and conclude by taking the proper limit. For that purpose let $N\in\bb{N}$, $\set{t_i}_{i=0}^N$ be a partition of $[0,T]$ with $t_0=0$ and $t_N=T$, $\set{\Phi_i}_{i=1}^{N}\subset L(V ,H)$ and an elementary process $\Phi$ defined as
	\begin{align*}
		\Phi(s) = \sum_{i=1}^N \Phi_i \ind{[t_{i-1},t_i)} (s).
	\end{align*}
	Then by using the linearity of the operators $\Phi_i$, definition \eqref{eq:w-n} and $Q^{1/2}f_k = \lambda^{1/2}f_k$,
	\begin{align*}
		\int_{0}^t \Phi_s dW_s^n &= \sum_{i=1}^N\Phi_i(W^n_{t_{i+1}\wedge t} - W^n_{t_i\wedge t}) = \sum_{i=1}^{N} \Phi_i\parent{\sum_{k=1}^n \sqrt{\lambda_k}f_k\beta^k_{t_{i+1}\wedge t} 
		-\sum_{k=1}^n \sqrt{\lambda_k}f_k\beta^k_{t_i\wedge t} }\\
		&=\sum_{i=1}^{N} \Phi_i\parent{\sum_{k=1}^n(Q^{1/2}f_k)(\beta^k_{t_i\wedge t} - \beta^k_{t_{i-1}\wedge t})}=\sum_{k=1}^n\sum_{i=1}^N\Phi_i(Q^{1/2}f_k)(\beta^k_{t_i\wedge t} - \beta^k_{t_{i-1}\wedge t})\\
		&=\sum_{k=1}^n \int_0^t \Phi_s (Q^{1/2}f_k)d\beta^k_s.
	\end{align*}
	It is easy to see that for every $j\in\bb{N}$, $(\Phi_s(Q^{1/2}f_j))_{s\in[0,T]}$ is an elementary process in $\mathscr{N}_W(0,T;L_2(\R,H))$; therefore, the property is satisfied for those processes. Now, given a sequence of elementary processes such that $\Phi^k\to\Phi$ in $\mathscr{N}_W(0,T;L_2(V_0,H))$, we also have that for every $j\in\bb{N}$ $\Phi^k(Q^{1/2}f_j)\to\Phi(Q^{1/2}f_j)$ in $\mathscr{N}_W(0,T;L_2(\R,H))$. For any $k\in\bb{N}$ it holds that,
	\begin{align*}
		\int_0^{\cdot} \Phi^k_s dW^N_s = \sum_{j=1}^n \int_0^{\cdot}\Phi^k_s(Q^{1/2}f_j) d\beta^j_s.
	\end{align*}
	The property follows by taking limit in $\mathscr{M}_T^2(H)$ as $k\to\infty$ in both sides.
\end{proof}
\begin{theorem}[\textit{Martingale Representation Theorem}]
	\label{lemma:mg-representation}
	Let $W$ be a Hilbert space and $r,s\in[0,T]$ with $r<s$. Then, for every $X\in L^2(\Omega,\ca{F}_s,\p;W)$ there exists $(Z_t)_{t\in[r,s]}\in \mathscr{N}_W([r,s];L^0_2(V,W))$ such that
	\begin{align*}
		X = \bb{E}(X|\ca{F}_t) + \int_t^s Z_u dW_u,\ t\in[r,s].
	\end{align*} 
\end{theorem}
	\begin{proof}
		See for instance \cite[Proposition $4.1$]{fuhrman}.

	\end{proof}

\section{The Forward-Backward Stochastic System}\label{sec:the-forward-backward-stochastic-system}

\subsection{Assumptions for the model}\label{sec:assumptions}

Recall the Kolmogorov model introduced in \eqref{eq:pde} and the Subsection \ref{sec:notation} (Notation) for details on the functional spaces. Along the paper we shall consider the following assumptions.
\begin{assumptions}\label{assumptions}
	There exists a constant $K>0$ such that,
	\begin{enumerate}
		\item {\bf Structure of $\ca{L}$}. The operator $\ca{L}$ is defined for $f\in C^{0,2}([0,T]\times H;\R)$ and $(t,x)\in[0,T]\times H$ as follows,
		\begin{align*}
			\ca{L}[f](t,x) = \prom{\nabla f(t,x), Ax + F(t,x)}_H + \frac{1}{2} \text{tr}\parent{\nabla^2 f(t,x)(B(t,x)Q^{1/2})(B(t,x)Q^{1/2})^*},
		\end{align*} 
		where 
		\begin{itemize}
		\item $\nabla f \in H$ is the standard gradient, and $\nabla^2 f$ is the bilinear operator second derivative;
		\item $A\colon\ca{D}(A)\subset H\to H$ is the infinitesimal generator of a $C_0$-semigroup $\set{S(t), t\ge 0}$ on $H$, with $\ca{D}(A)$ dense in $H$ and $x\in \ca{D}(A)$. 
		%\item 
		\item $F$ is a drift term and $B$ is an diffusion operator satisfying
		\[
		F\colon[0,T]\times H\to H, \qquad  B\colon[0,T]\times H\to L_2(V_0,H),
		 \] 
		 are $(\ca{B}([0,T])\otimes\ca{B}(H))$-$\ca{B}(H)$ and $(\ca{B}([0,T])\otimes \ca{B}(H))$-$\ca{B}(L_2(V_0,H))$ measurable mappings, respectively. Furthermore, they satisfy that for all $x,y\in H$ and $t\in[0,T]$,
		\begin{align*}
			\norm{F(t,x) - F(t,y)}_H + \norm{B(t,x) - B(t,y)}_{L_2(V_0,H)}\le K\norm{x-y}_H,
		\end{align*}
		and
		\begin{align*}
			\norm{F(t,x)}^2_{H} + \norm{B(t,x)}^2_{L_2(V_0,H)} \le K^2(1+\norm{x}^2_H).
		\end{align*}
		These mean that $F$ and $B$ are uniformly Lipschitz, with linear growth. 
		\smallskip
		\item For all $r,s\in[0,T]$ with $r < s$ and $y\in H$, 
		\[
		S(s-r)F(r,y)\in\ca{D}(A), \quad S(s-r)B(r,y)\in\ca{D}(A).
		\]
		And, there exists positive functions $g_1, g_2\in L^1([0,T])$ such that
		\begin{align*}
			\norm{AS(s-r)F(r,y)}_H &\le g_1(s-r)\parent{1 + \norm{y}_H},\\
			\norm{AS(s-r)B(r,y)}^2_{L_2(V_0,H)} &\le g_2(s-r)\parent{1 + \norm{y}^2_H}.
			\end{align*}
		\end{itemize}
		Note that this tells us that $F$ and $B$ are uniformly bounded in $[0,T]$ for fixed $x\in H$. We also denote as $B^*$ the adjoint operator of $B$.
		
		\item  {\bf Structure of the nonlinearity}. $\psi\colon[0,T]\times H\times\R\times V \to\R$ is the nonlinearity in \eqref{eq:pde}, which satisfies that for $t,t'\in [0,T], x,x'\in H, y,y'\in\R$ and $z,z'\in V$,
		\begin{align}\label{lip:psi}
			|\psi(t,x,y,z) - \psi(t',x',y',z')| \le C(|t-t'|^{1/2} + \norm{x-x'}_H + |y-y'| + \norm{z-z'}_V).
		\end{align}
	\end{enumerate}
\end{assumptions}

These assumptions are standard in the literature, see e.g. \cite{DBS}. In particular, condition \eqref{lip:psi} on $\psi$ is required to control our numerical scheme in a satisfactory way. As for the conditions on $\mathcal L$, these are also common in the infinite dimensional literature, as expressed for example in \cite{fuhrman}. For any $u\in V$ we have that $\norm{Q^{1/2}u}_V\le\norm{Q^{1/2}}_{L(V)}\norm{u}_V=\norm{Q^{1/2}}_{L(V)}\norm{Q^{1/2}u}_0$, which will be implicitly used during the paper.

\subsection{The forward process} Now we recall the mathematical structure associated to the forward process $(X_t)$ in \eqref{eq:fpsde}, where $A$, $B$ and $F$ were specified in Assumptions \ref{assumptions}. For further details, the reader can consult \cite{daprato}.

\begin{definition}[Strong and mild solutions]~
\label{def:strong-mild}
\begin{enumerate}
\item	A predictable $H$-valued stochastic process $(X_t)_{t\in[0,T]}$ is said to be a {\bf strong solution} of \eqref{eq:fpsde} if for all $t\in[0,T]$ $X_t\in \ca{D}(A)$ $\p$-a.e.,
	\begin{align*}
		\int_0^T \norm{AX_s}_H ds < \infty,\quad \p\text{-a.e.}
	\end{align*}
	and equation \eqref{eq:fpsde} is satisfied for all $t\in[0,T]$.

\item	A predictable $H$-valued stochastic process $(X_t)_{t\in[0,T]}$ is said to be a {\bf mild solution} of \eqref{eq:fpsde} if
	\begin{align*}
		\p\parent{\int_0^T \norm{X_s}_H^2 ds< \infty} = 1,
	\end{align*}
	and for all $t\in[0,T]$ we have the weak formulation of \eqref{eq:fpsde}:
	\begin{align}\label{eq:mild}
		X_t = S(t) x + \int_0^t S(t-s) F(s,X_s)ds + \int_0^t S(t-s)B(s,X_s)dW_s,\quad\p\text{-a.e.}
	\end{align}
\end{enumerate}	
\end{definition}

The following result gives existence of mild solutions in a very general setting.

\begin{theorem}\label{theorem:mild-existence-uiqueness}
	There exist a unique mild solution $(X_t)_{t\in[0,T]}$ to \eqref{eq:fpsde}, unique among the stochastic processes satisfying,
	\begin{align*}
		\p\parent{\int_0^T \norm{X_s}_H^2 ds< \infty} = 1.
	\end{align*}
	Moreover, $X$ possesses a continuous modification and for any $p\ge 2$ there exists a constant $C=C(p,T)>0$ such that,
	\begin{align*}
		\underset{s\in[0,T]}{\sup} \bb{E}\norm{X_s}_H^p \le C(1+\norm{x}^p_H).
	\end{align*}
	\begin{proof}
		See \cite[Theorem $7.2$]{daprato}.
	\end{proof}
\end{theorem}

Now we provide a proof of existence of strong solutions to  \eqref{eq:fpsde}, which follows closely  \cite[Theorem $2$]{albeverio}.

\begin{prop}\label{prop:strong-forward-solution}
	Assuming Assumptions \ref{assumptions} there exists a strong solution $(X_t)_{t\in[0,T]}$ to the equation \eqref{eq:fpsde} and $C=C(T)$ such that
	\begin{align}\label{eq:mild-bound}
		\underset{s\in[0,T]}{\sup}\bb{E}\norm{X_s}_H^2 \le C\quad\text{and}\quad\p\parent{\int_0^T \norm{X_s}_H^2 ds< \infty} = 1.
	\end{align}
\end{prop}	
	\begin{proof}
By applying Theorem \ref{theorem:mild-existence-uiqueness} we have a mild solution already satisfying \eqref{eq:mild-bound} and then, due to Assumptions \ref{assumptions}, from \eqref{eq:mild} we get that for all $t\in[0,T]$, $X_t\in\ca{D}(A)\ \p$-a.e. and
		\begin{align*}
			\int_0^t AX_s = \int_0^t AS(s)xds + \underbrace{\int_0^t\int_0^sAS(s-r)F(r,X_r)drds}_{\bf I} + \underbrace{\int_0^t\int_0^sAS(s-r)B(r,X_r)dW_rds}_{\bf II}.
		\end{align*}
		Basically, the idea here is to use Fubini theorem and its stochastic version (see \cite[Section $4.5$]{daprato}) together with the fact that $S(t)y - y = \int_0^t AS(s)ds$ for $y\in\ca{D}(A)$. The bounds that $F$ and $B$ satisfy in Assumptions \ref{assumptions} imply that,
		\begin{align*}
			\int_0^T\int_0^s \norm{AS(s-r)F(r,X_r)}_H drds &\le \int_0^T\int_0^s g_1(s-r)drds + \int_0^T\int_0^s g_1(s-r)\norm{X_r}_H drds \\
			&\le \norm{g_1}_{L^1([0,T])}\parent{T + \int_0^T \norm{X_r}_H dr}< \infty\quad \p\text{-a.e.}.
		\end{align*}
		And,
		\begin{align*}
			\int_0^T \bb{E}\int_0^s\norm{AS(s-r)B(r,X_r)}^2_{L_2(V_0,H)} drds&\le \int_0^T\int_0^s g_2(s-r)drds + \int_0^T\bb{E}\int_0^s g_2(s-r)\norm{X_r}^2_H drds\\
			&\le \norm{g_1}_{L^1([0,T])}\parent{1 + T\bb{E}\Bigg[\underset{r\in[0,T]}{\sup}\norm{X_r}^2_H\Bigg]} < \infty.
		\end{align*}
		Then, by Fubini Theorem,
		\begin{align*}
			{\bf I} &= \int_0^t S(t-r)F(r,X_r) dr - \int_0^t F(r,X_r) dr\quad\text{and},\\
			{\bf II} &= \int_0^t S(t-r)B(r,X_r) dW_r - \int_0^t B(r,X_r) dW_r.
		\end{align*}
		Therefore,
		\begin{align*}
			\int_0^t AX_s ds = &~{} S(t)x - x + \int_0^t S(t-r)F(r,X_r) dr - \int_0^t F(r,X_r) dr \\
			& + \int_0^t S(t-r)B(r,X_r) dW_r - \int_0^t B(r,X_r) dW_r.
		\end{align*}
		Hence,
		\begin{align*}
			X_t = x + \int_0^t AX_s ds + \int_0^tF(r,X_r)dr + \int_0^t B(r,X_r)dW_r\quad\p\text{-a.e.},
		\end{align*}
		and the proof is complete.
	\end{proof}

\subsection{The backward process}

Now we provide existence results for the backward process \eqref{eq:bpsde}, following ideas in \cite[Lemma $4.2$]{fuhrman}.

\begin{lemma}\label{lemma:representation-existence}
	Let $\eta\in L^2(\Omega,\ca{F}_T,\p)$ and $f\in\mathscr{N}_W (0,T;\R)$. Then there exist a unique pair $(Y,Z)\in \mathscr{S}^2_T(\R)\times\mathscr{N}_W (0,T;L_2^0(V,\R))$ such that,
	\begin{align}
		\label{eq:lemma-backward-eq}
		Y_t = \eta + \int_t^T f_s ds - \int_t^T \prom{Z_s,\cdot}_0 dW_s.
	\end{align}
	Furthermore, the following bounds are satisfied,
	\begin{align}
		\label{eq:bounds-lemma}
		\bb{E} \left( \int_0^T e^{2\beta s}\norm{Z_s}_0^2 ds \right) \wedge \bb{E}\left( \underset{s\in[0,T]}{\sup} e^{2\beta s} |Y_s|^2 \right) \le  \frac{4}{\beta}\bb{E}\int_0^T e^{2\beta s} |f_s|^2 ds + 8e^{2\beta T} \bb{E}|\eta|^2.
	\end{align}
	Where $\wedge$ indicates the maximum between both quantities.
\end{lemma}
\begin{proof}
	For uniqueness to the first part of \cite[Lemma $4.2$]{fuhrman}. First, we prove existence, define $\xi=\eta ~ + ~ \int_0^T f_s ds\in L^2(\Omega,\ca{F}_T,\p)$. Then, by Theorem \eqref{lemma:mg-representation}, there exists $Z\in\mathscr{N}_W (0,T;L_2^0(V,\R))$ such that
	\begin{align}\label{eq:chi-representation}
		\xi = \bb{E}(\xi|\ca{F}_t) + \int_t^T \prom{Z_s,\cdot}_0 dW_s,
	\end{align}
	where we applied Remark \ref{remark:schmidt-dual} to notice that $L_2(V_0,\R)=V_0$. Define now $Y_t=\bb{E}(\xi|\ca{F}_t) - \int_0^t f_s ds$, follows that
	\begin{align}\label{eq:representation-result}
		Y_t = \eta + \int_t^T f_s ds - \int_t^T \prom{Z_s,\cdot}_0 dW_s.
	\end{align}
	To conclude that $(Y_t)_{t\in[0,T]}\in\mathscr{S}^2_T(\R)$ we just note that by \eqref{eq:chi-representation}, \eqref{eq:representation-result} and the definition of $\xi$, one has for every $t\in[0,T]$
	\begin{align*}
		\bb{E}|Y_t|^2\le 3\parent{\bb{E}|\eta|^2 + T\bb{E}\int_0^T |f_s|^2 ds + \bb{E}\int_0^T \norm{Z_s}_0^2 ds}\le 27\parent{\bb{E}|\eta|^2 + \bb{E}\int_0^T f_s^2 ds }<\infty. 
	\end{align*}
	In order to prove estimate \eqref{eq:bounds-lemma}, we bound both quantities at left side by the right side. Sssume the existence and uniqueness of a solution $(Y,Z)$ and note that for almost all $s\in[0,T]$, $\bb{E}|f_s|^2<\infty$, thus by Theorem \ref{lemma:mg-representation} there exists $(K(u,s))_{u\in[0,s]}\in\mathscr{N}_W(0,s;L_2^0(V,\R))$ such that,
	\begin{align}
		\label{eq:f}
		f_s = \bb{E}(f_s|\ca{F}_t) + \int_t^s K(u,s)dW_u,\ t\in[0,s].
	\end{align}
	We extend $K$ to $[0,T]\times [0,T]$ in the following way,
	\begin{align*}
		K:[0,T]\times[0,T]\times\Omega&\longrightarrow L_2^0(V,\R)\\
		(u,s,\omega) \qquad &\longmapsto ~ K(u,s)(\omega)\ind{[0,s]}(u)=\begin{cases}K(u,s)(\omega),\quad &u\le s\\ 0,\quad&\sim. \end{cases}
	\end{align*}
	$(\ca{P}_T\times\ca{B}([0,T]))$-measurability of $K$ is discussed in \cite{fuhrman}, but it is no difficult to convince oneself of this. In the same way there exists $(L_t)_{t\in[0,T]}\in\mathscr{N}_W(0,T;L_2^0(V,\R))$ such that,
	\begin{align}
		\label{eq:eta}
		\eta = \bb{E}(\eta|\ca{F}_t) + \int_t^T L_sdW_s,\quad  t\in[0,T].
	\end{align}
	By taking $\bb{E}(\cdot|\ca{F}_t)$ in \eqref{eq:lemma-backward-eq} then using conditional Fubini's theorem, and replacing \eqref{eq:f} and \eqref{eq:eta} we have that for all $t\in[0,T]$,
	\begin{align*}
		Y_t = \eta - \int_t^T f_sds - \int_t^T L_s dW_s + \int_t^T \int_t^T K(u,s)\ind{[t,s]}(u) dW_u ds. 			
	\end{align*}
	Due to $\int_t^T \bb{E}\int_t^T \norm{K(u,s)}_0^2\ind{[t,s]}(u) duds < \infty$ (it can be bounded by a factor of $\norm{f}_{\mathscr{N}}$), we may apply stochastic Fubini theorem (see \cite[Section $4.5$]{daprato}) getting,
	\begin{align*}
		Y_t = \eta -\int_t^T f_s ds - \int_t^T\parent{L_u - \int_u^T K(u,s) ds}dW_s.
	\end{align*}
	Then by uniqueness,
	\begin{align*}
		Z_u = L_u - \int_u^T K(u,s) ds,\quad\forall u\in[0,T],
	\end{align*}
	which allows us to compute,
	\begin{align*}
		\bb{E}\int_0^T e^{2\beta u}\norm{Z_u}^2_0 du = \underbrace{2\bb{E}\int_0^T e^{2\beta u} \norm{L_u}^2_0 du}_{\bf I} + \underbrace{2\bb{E}\int_0^T e^{2\beta u}\norm{\int_u^T K(u,s)ds}_0^2 du}_{\bf II}.
	\end{align*}
	By standard procedures and using \eqref{eq:eta} we get ${\bf I}\le 8e^{2\beta T}\bb{E}|\eta|^2$. To work with ${\bf II}$ we first note that for any $u\in [0,T]$,
	\begin{align*}
		\norm{\int_u^T K(u,s) ds}_0^2 \le \int_u^T e^{-2\beta s}ds \int_u^T e^{2\beta s}\norm{K(u,s)}_0^2 ds\le\frac{e^{-2\beta u}}{2\beta} \int_u^T e^{2\beta s}\norm{K(u,s)}_0^2 ds,
	\end{align*}
	where we applied Bochner's estimate ($\norm{\int f}\le \int \norm{f}$) and H\"older's inequality. Then, by replacing the last relation in ${\bf II}$ and using Fubini theorem, 
	\begin{align*}
		{\bf II}&\le \frac{1}{\beta}\bb{E} \int_0^T\int_u^T e^{2\beta s}\norm{K(u,s)}_0^2 ds du = \frac{1}{\beta}\bb{E} \int_0^T\int_0^T e^{2\beta s}\norm{K(u,s)}_0^2 \ind{[u,T]}(s) ds du\\
		& = \frac{1}{\beta} \int_0^T e^{2\beta s} \bb{E}\parent{\int_0^s \norm{K(u,s)}_0^2 du} ds \le \frac{4}{\beta}\int_0^T e^{2\beta s}\bb{E}|f_s|^2 ds
	\end{align*}
	Now for the second bound we first note that by taking $\bb{E}(\cdot|\ca{F}_t)$ we have, 	
	\begin{align*}
		Y_t = \bb{E}(\eta|\ca{F}_t) - \bb{E}\parent{\int_t^T f_s ds\Bigg|\ca{F}_t},
	\end{align*}
	and then,
	\begin{align*}
		\bb{E} \underset{t\in[0,T]}{\sup} e^{2\beta t} |Y_t|^2 \quad  \le \quad  \underbrace{2\bb{E} \underset{t\in[0,T]}{\sup} e^{2\beta t} |\bb{E}(\eta|\ca{F}_t)|^2}_{\bf A} \quad +\quad  \underbrace{2\bb{E} \underset{t\in[0,T]}{\sup} e^{2\beta t}\Bigg| \bb{E}\parent{\int_t^T f_s ds \Bigg|\ca{F}_t} \Bigg|^2}_{\bf B}.
	\end{align*}
	Using Doob's inequality we get ${\bf A}\le 8e^{2\beta T}\bb{E}|\eta|^2$. For the second term,
	\begin{align*}
		{\bf B} &\le 2\bb{E} \underset{t\in[0,T]}{\sup} e^{2\beta t}\Bigg| \bb{E}\parent{\sqrt{\int_t^T e^{-2\beta s} ds} \sqrt{\int_t^T e^{2\beta s} |f_s|^2 ds} \Bigg|\ca{F}_t} \Bigg|^2\\
		&\le \frac{1}{\beta}\bb{E} \underset{t\in[0,T]}{\sup}\Bigg| \bb{E}\parent{ \sqrt{\int_0^T e^{2\beta s} |f_s|^2 ds} \Bigg|\ca{F}_t} \Bigg|^2\\
		&\le \frac{4}{\beta}\bb{E}\int_0^T e^{2\beta s}|f_s|^2 ds.
	\end{align*}
	Where we used Doob's inequality on the last inequality. By putting all together we conclude the proof.
\end{proof}
\subsection{Existence in the nonlinear Forward-Backward model}
The existence and uniqueness of a solution $(Y,Z)$ to the backward equation \eqref{eq:bpsde} is well-known, here we follow the proof given in \cite{fuhrman}. The argument, as we are working in a non-linear framework, relies on an application of Banach's fixed point theorem. The problem is that with the parameters as they are, the fixed-point functional does not necessarily contract. A solution to this issue is possible by giving equivalent norms to $\mathscr{N}_W(0,T;L_2^0(V;\R))$ and $\mathscr{S}^2_T(\R)$ parameterized by a positive real number $\beta$. Let $\beta>0$, consider
\begin{align*}
	\norm{Y}^2_{\mathscr{S}^2_{T,\beta}} = \bb{E} \left( \underset{s\in[0,T]}{\sup} e^{2\beta s}|Y|^2 \right) \quad\text{and}\quad\norm{Z}^2_{\mathscr{N}_{W,\beta}} = \bb{E} \left(\int_0^T e^{2\beta s}\norm{Z}^2_0 ds\right).
\end{align*} 
With a bit of work we can see that $\norm{\cdot}_{\mathscr{S}^2_{T,\beta}}$ and $\norm{\cdot}_{\mathscr{N}_{W,\beta}}$ are equivalent to $\norm{\cdot}_{\mathscr{S}^2_T}$ and $\norm{\cdot}_{\mathscr{N}_W}$, respectively.
\begin{prop} 
	\label{prop:backward-existence}
	Given a $H$-valued stochastic process $(X_t)_{t\in[0,T]}$ such that
	\begin{align}\label{eq:bound-backward-existence}
		\bb{E} \left( \int_0^T \psi(s,X_s,0,0)^2 ds \right)< \infty,
	\end{align}
	there exist a unique solution $(Y,Z)\in\mathscr{S}^2_T(\R)\times\mathscr{N}_W (0,T;L_2^0(V,\R))$ to equation \eqref{eq:bpsde} and there exists $C=C(K,T)>0$ such that,
	\begin{align}\label{eq:bound2-backward-existence}
		\norm{Y}_{\ca{S}^2_T}^2 + \norm{Z}_{\ca{N}_W}^2 \le C \parent{\bb{E}\phi(X_T)^2 + \bb{E}\int_0^T \psi(s,X_s,0,0)^2ds}.
	\end{align}
	\begin{proof}
		Again, we follow the proof given in \cite[Proposition $4.3$]{fuhrman}. The following result is proven as the majority of existence of solutions to non-linear equations results, this is, by considering an adequate operator from a Banach space to itself and applying Banach's fixed point Theorem. For $\beta>0$ consider $\mathscr{K}_{\beta} = \mathscr{S}^2_T(\R)\times\mathscr{N}_W (0,T;L_2^0(V,\R))$ which is a Banach space endowed with,
		\begin{align*}
			\norm{(Y,Z)}_{\mathscr{K}_{\beta}}^2 =&\norm{Y}^2_{\mathscr{S}^2_{T,\beta}} + \norm{Z}^2_{\mathscr{N}_{W,\beta}}\\
			=&~{} \bb{E}\underset{s\in[0,T]}{\sup}e^{2\beta s} |Y_s|^2 + \bb{E}\int_0^T e^{2\beta s}\norm{Z_s}^2_0 ds.
		\end{align*}
		Let $\Psi\colon\mathscr{K}_{\beta}\to\mathscr{K}_{\beta}$ be defined as $\Psi(U,V)=(Y,Z)$ where $(Y,Z)$ is such that,
		\begin{align*}
			Y_t + \int_t^T \prom{Z_s,\cdot}_0 dW_s = \phi(X_T) + \int_t^T\psi(s,X_s,U_s,V_s)ds.
		\end{align*}
		Given $(U,V)\in\mathscr{K}_{\beta}$, $\Psi(U,V)$ is well-defined by Lemma \ref{lemma:representation-existence} taking $(f_s)_{s\in[0,T]} = (\psi(s,X_s,U_s,V_s))_{s\in[0,T]}$ which is an element of $\mathscr{N}_W(0,T;\R)$ due to the Lipschitz condition imposed on $\psi$ and \eqref{eq:bound-backward-existence}, the existence is proven if we show that $\Psi$ is a contraction. Let $(U,V),(\bar{U},\bar{V}),(Y,Z),(\bar{Y},\bar{Z})\in\mathscr{K}_{\beta}$ be such that $\Psi(U,V) = (Y,Z)$ and $\Psi(\bar{U},\bar{V}) = (\bar{Y},\bar{Z})$, follows that for all $t\in [0,T]$,
		\begin{align*}
			Y_t - \bar{Y}_t + \int_t^T \prom{Z_t - \bar{Z}_t,\cdot}_0 dW_S = \int_t^T\parent{\psi(s,X_s,U_s,V_s) - \psi(s,X_s,\bar{U}_s,\bar{V}_s)} ds.
		\end{align*}
		This means that $(Y-\bar{Y},Z-\bar{Z})$ satisfies Lemma \ref{lemma:representation-existence} with $\eta = 0$ and $f_s = \psi(s,X_s,U_s,V_s) - \psi(s,X_s,\bar{U}_s,\bar{V}_s)$. Thus
		\begin{align*}
			\norm{\Psi(U,V) - \Psi(\bar{U},\bar{V})}^2_{\mathscr{K}_{\beta}}&\le \frac{8K}{\beta}\bb{E}\int_0^T e^{2\beta s}\parent{|U_s-\bar{U}_s|^2+\norm{V_s-\bar{V}_s}_0^2} ds\\
			&\le\frac{8K}{\beta}\bb{E}\parent{T\underset{s\in[0,T]}{\sup}e^{2\beta s}|U_s-\bar{U}_s|^2 + \int_0^T e^{2\beta s}\norm{V_s-\bar{V}_s}^2_0 ds }\\
			&\le\frac{8K(T + 1)}{\beta}\norm{(U,V) - (\bar{U},\bar{V})}^2_{\mathscr{K}_{\beta}}.
		\end{align*}
		By taking $\beta=17K(T+1)$ we show that $\Psi$ is a contraction, and therefore, the existence is proven. Uniqueness follows easily by standard arguments. Consider now the solution $(Y,Z)$ by estimates \eqref{eq:bounds-lemma},
		\begin{align*}
			\norm{Y}_{\ca{S}^2_{T,\beta}}^2 + \norm{Z}_{\ca{N}_{W,\beta}}^2 &\le 16e^{2\beta T}\bb{E}\phi(X_T)^2 + \frac{8}{\beta}\underbrace{\bb{E}\int_0^T \psi(s,X_s,Y_s,Z_s)^2 ds}_{{\bf I}}.
		\end{align*}
		Now, by the Lipschitz condition,
		\begin{align*}
			{\bf I} &\le 2K\bb{E}\int_0^T e^{2\beta s} \parent{|Y_s|^2 + \norm{Z_s}_0^2} + 2e^{2\beta T}\bb{E}\int_0^T \psi(s,X_s,0,0)^2ds\\
			&\le 2K(T+1)\parent{\norm{Y}_{\ca{S}^2_{T,\beta}}^2 + \norm{Z}_{\ca{N}_{W,\beta}}^2} + 2e^{2\beta T}\bb{E}\int_0^T \psi(s,X_s,0,0)^2ds.
		\end{align*}
		Hence,
		\begin{align*}
			\norm{Y}_{\ca{S}^2_{T,\beta}}^2 + \norm{Z}_{\ca{N}_{W,\beta}}^2 &\le 16e^{2\beta T}\bb{E}\phi(X_T)^2 + \frac{16}{\beta}e^{2\beta T}\bb{E}\int_0^T \psi(s,X_s,0,0)^2ds \\
			&+ \frac{16K(T+1)}{\beta} \parent{\norm{Y}_{\ca{S}^2_{T,\beta}}^2 + \norm{Z}_{\ca{N}_{W,\beta}}^2}.
		\end{align*}
		Chosen $\beta$ ensure that $16K(T+1)/\beta < 1$ and therefore,
		\begin{align*}
			\norm{Y}_{\ca{S}^2_{T}}^2 + \norm{Z}_{\ca{N}_{W}}^2 &\le \norm{Y}_{\ca{S}^2_{T,\beta}}^2 + \norm{Z}_{\ca{N}_{W,\beta}}^2\\
			&\le \Big[1 - \frac{16K(T+1)}{\beta}\Big]^{-1}\parent{16e^{2\beta T}\bb{E}\phi(X_T)^2 + \frac{16}{\beta}e^{2\beta T}\bb{E}\int_0^T \psi(s,X_s,0,0)^2ds}.
		\end{align*}
		Hence, estimate \eqref{eq:bound2-backward-existence} follows. The method that we have used remains valid if we intend to prove the existence of solutions $(Y,Z)\in\mathscr{S}^2_T(K)\times\mathscr{N}_W(0,T;L^0_2(V,K))$ and $\psi,\phi$ also taking values in the Hilbert space $K$.  
	\end{proof} 
\end{prop}
Previous proposition lets us state, given our assumptions \eqref{assumptions}, that from now on we can refer to a solution $(X,Y,Z)$ of the system \eqref{eq:fpsde}-\eqref{eq:bpsde} with $(Y,Z)\in\mathscr{S}^2(\R)\times\mathscr{N}_W(0,T;L_2^0(V,\R))$ and $X$ a strong solution of the forward equation \eqref{eq:fpsde} given by Proposition \ref{prop:strong-forward-solution}. 

\subsection{Extra bounds on the nonlinear part} Finally, we finish this section with a boundedness lemma.

\begin{lemma}\label{lemma:f-int-bound}
	Let $(X_t)_{t\in[0,T]}$ be such that $\underset{s\in[0,T]}{\sup}\bb{E}\norm{X_s}_H^2 < \infty$ and $(Y,Z)\in\mathscr{S}^2_T(\R)\times\mathscr{N}_{W}(0,T;L_2^0(V))$. The following bound holds,
	\begin{align*}
		\bb{E} \left( \int_0^T \psi(s,X_s,Y_s,Z_s)^2 ds  \right) < \infty
	\end{align*}
	\begin{proof}
		First note that
		\begin{align*}
			\int_0^T\bb{E}\norm{X_s}_H^2 ds \le T\underset{s\in[0,T]}{\sup}\bb{E}\norm{X_s}_H^2 < \infty,
		\end{align*}
		then, Fubini theorem can be applied together with the Lipschitz condition on $\psi$ to get,
		\begin{align*}
			\bb{E}\int_0^T \psi(s,X_s,Y_s,Z_s)^2 ds&\le 2K\bb{E}\int_0^T (s + \norm{X_s}_H^2 + |Y_s|^2 + \norm{Z_s}_0^2)ds + 2T\psi(0,0,0,0)^2\\
			& \le \frac{CT^2}{2} + CT\underset{s\in [0,T]}{\sup}\bb{E}\norm{X_s}^2_H + CT\underset{s\in [0,T]}{\sup} |Y_s|^2 + C\bb{E}\int_0^T\norm{Z_s}_0^2 ds+CT\\
			& \le C\parent{1+\underset{s\in [0,T]}{\sup}\bb{E}\norm{X_s}^2_H + \norm{Y}^2_{\mathscr{S}^2_T(\R)} + \norm{Z}^2_{\mathscr{N}_W(0,T;L_2^0(V,\R))}}<\infty.
		\end{align*}
		Thus, the proof is completed.
	\end{proof}
\end{lemma}

\section{Functional Numerical Scheme}\label{sec:Functional Numerical Scheme}
Throughout this section we will work with functions that we call \textit{approximators} and are parameterized by a finite dimensional parameter $\theta\in\Theta_{\eta}\subset\R^{\eta}$ for some $\eta\in
\bb{N}$, also let $\Theta=\cup_{\eta\in\bb{N}}\Theta_{\eta}$. As the reader may anticipate, these functions will be the DeepOnets introduced in Section \ref{sec:infinite-dim-nn}. We work in generality first, to then apply our results to this particular case. %We adapt this level of generality due to the fact that all the computation and results derived here are independent of the form that the approximation may take.

\medskip

The following is a key assumption for the validity of our main results. 

\begin{assumptions}\label{assumptions3}
Assume we are given a function $u\in C^{1,2}([0,T]\times H)$ satisfying (\ref{eq:pde}) and a strong solution $(X_t)_{t\in [0,T]}$ to (\ref{eq:fpsde}).
\end{assumptions}

This assumption is natural in finite dimensions, but its validity in infinite dimensions is far from obvious.

\subsection{The numerical scheme}\label{sec:numerical-scheme}
 The scheme presented here is fully inspired by \cite{DBS} and relies on an application of It\^o Lemma to $(u(t,X_t))_{t\in [0,T]}$ as follows (see \cite[Theorem $4.32$]{daprato}),
\begin{align*}
	& u(t,X_t) \\
	&= u(0,X_0) + \int_0^t \prom{\nabla u(s,X_s), B(s,X_s)(\cdot)}_H dW_s - \int_0^t \psi\parent{s,X_s,u(s,X_s),B^* (s,X_s)\nabla u(s,X_s)} ds\\
	&=u(0,X_0) + \int_0^t \prom{B^*(s,X_s)\nabla u(s,X_s), \cdot}_{0} dW_s - \int_0^t \psi\parent{s,X_s,u(s,X_s),B^* (s,X_s)\nabla u(s,X_s)} ds.
\end{align*}
Consider now a uniform partition $\pi=\set{t_0=0,...,t_N=T}$ with $t_i=\frac{iT}{N}$ such that $h = t_{i+1} - t_i>0$ for all $i\in\set{0,...,N-1}$, then
\begin{align*}
	u(t_{i+1},X_{t_{i+1}}) =&~{} u(t_i,X_{t_i}) + \integral \prom{B^*(s,X_s)\nabla u(s,X_s), \cdot}_0 dW_s\\
	&- \integral \psi\parent{s,X_s,u(s,X_s),B^*(s,X_s)\nabla u(s,X_s)}ds.
\end{align*}
Let $\eta\in\bb{N}$ be a fixed natural number and let $\Theta_{\eta}\subset\R^{\eta}$ be also a fixed set. Now, let us introduce some \textit{approximators} as a collection of mappings $u_i^{\theta}\colon H\to\R$ for $i\in\set{0,...,N}$ and $z_i^{\theta}\colon H\to V_0 $ for $i\in\set{0,...,N-1}$. Additionally, consider an scheme $X^{\pi}=(X^{\pi}_t)_{t\in\pi}$ for the equation (\ref{eq:fpsde}) which we assume satisfies $\sigma(X^{\pi}_s\colon s\le t,s\in\pi)\subset \ca{F}_{t}$, $X^{\pi}_t\in L^4(\Omega,\ca{F}_t,\p;H)$ for $t\in\pi$. Here $X^{\pi}$ is a Markov process. These approximators are assumed to be such that $\set{u_i^{\theta}}_{\theta\in\Theta}$ and $\set{z_i^{\theta}}_{\theta\in\Theta}$ are dense in $L^2(H,\mu_{\xscheme{i}})$ and $L^2(H,\mu_{\xscheme{i}};V_0)$ respectively. Also assume that the approximators has polynomial growth at most. 
\begin{remark}
	Hilbert valued DeepOnets are a set of \textit{approximators}. This is obtained by defining
	\begin{align}
		\Theta_{\eta} = \bigcup_{d,m\in\bb{N}} \set{d}\times\ca{N}_{\sigma,7,d,m,\eta}\times\set{m}.
	\end{align}
	The size of the hidden layers of the NN (recall Definition \ref{def:finite-dim-nn}) is the variable that may increase in order to have a better performance of the DO.
\end{remark}
We propose a scheme in which we intend to find $\theta\in\Theta_{\eta}$ such that given $\hat{u}_{i+1}$, the following approximations hold as good as possible:
\begin{align*}
	u_i^{\theta}(\cdot)&\approx u(t_i,\cdot)\\
	z_i^{\theta}(\cdot)&\approx B^*(t_i,\cdot)\nabla u(t_i,\cdot)\\
	\hat{u}_{i+1}(\xscheme{i+1})&\approx u_i^{\theta}(\xscheme{i}) + \integral\prom{z^{\theta}_i(\xscheme{i}), \cdot}_0 dW_s - \psi\parent{t_i,\xscheme{i},u^{\theta}_i(\xscheme{i}),z^{\theta}_i(\xscheme{i})}h,
\end{align*}
each one in some proper measure for every $i\in\set{1,...,N-1}$. The above approximations motivates the definition of a cost function, $L_i:\Theta_{\eta}\to[0,+\infty)$, associated to $\theta\in\Theta_{\eta}$:
\begin{align*}
	L_i(\theta) = \bb{E}\Big| \hat{u}_{i+1}(\xscheme{i+1}) - u_i^{\theta}(\xscheme{i}) - \integral\prom{z^{\theta}_i(\xscheme{i}), \cdot}_{0}dW_s + \psi\parent{t_i,\xscheme{i},u^{\theta}_i(\xscheme{i}),z^{\theta}_i(\xscheme{i})}h \Big|^2.
\end{align*}
We present the following algorithm as an infinite-dimension extension of the one already presented in \cite{DBS} and \cite{yo}.\\

\medskip

\begin{algorithm}[H]
	\label{algorithm}
	\SetAlgoLined
	Start with $\hat{u}_N=\phi$\;
	\For{$i\in\set{N-1,...,1}$}{
		Given $\ug{i+1}$\;
		Compute $\theta^* = \underset{\theta\in\Theta_{\eta}}{\text{argmin}}\ L_i(\theta)$\;
		Update $(\ug{i},\hat{z}_{i})=\parent{u^{\theta^*}_i,z^{\theta^*}_i}$\;
	}
	\caption{DBDP1 infinite-dimension extension}
\end{algorithm}

\subsection{Previous Definitions and Results}
Let us introduce the operator $\bb{E}_i = \bb{E}(\cdot|\ca{F}_{t_i})$ defined for every integrable real or vector valued random variable. For the consistency proof of the algorithm we need to introduce a somehow auxiliary scheme $(\widehat{\ca{V}}_{t_i},\gb{Z}_{t_i} )_{i\in\set{0,...,N-1}}$ that is inspired by \cite{bruno-touzi}, used in \cite{DBS} and we generalize to the infinite-dimensional case as follows,
\begin{align}
	\widehat{\ca{V}}_{t_i} &= \bb{E}_i (\hat{u}_{i+1}(\xscheme{i+1})) + \psi(t_i,\xscheme{i},\widehat{\ca{V}}_{t_i},\gb{Z}_{t_i})h
	\label{eq:nu-gorro}\\
	\gb{Z}_{t_i} &= \frac{1}{h}\bb{E}_i (\hat{u}_{i+1}(\xscheme{i+1})\Delta W_i).\label{eq:z-gorro-barra}
\end{align}
Observe that these processes are adapted to the discrete filtration $\parent{\ca{F}_t}_{t\in\pi}$. The discrete process $\widehat{\ca{V}}_{t_i}$ for $i\in\set{0,...,N-1}$ is well-defined for sufficiently small $h$ as shown in Lemma \ref{lemma:vwell-defined} and by Markov property of $X^{\pi}$, there exists square integrable functions $v_i,z_i$ for $i\in\set{0,...,N-1}$ such that
\begin{align*}
	\widehat{\ca{V}}_{t_i} = v_i(\xscheme{i})\quad\text{and}\quad \gb{Z}_{t_i} = z_i(\xscheme{i}).
\end{align*}
\begin{lemma}\label{lemma:vwell-defined}
	Assume that for sufficiently small $h$ and every $i\in\set{0,...,N-1}$, $\bb{E}|\hat{u}_{i+1}(\xscheme{i+1})|^4<+\infty$. Then there exists $\widehat{\ca{V}}_{t_i}\in L^2 (\Omega, \ca{F}_{t_i}, \p)$ such that (\ref{eq:nu-gorro}) holds and $\gb{Z}_{t_i}\in L^2 (\Omega,\ca{F}_{t_i},\p;V)$.
\end{lemma}
	
	\begin{proof}
		Let $i\in\set{0,...,N-1}$ and $f:L^2(\Omega,\ca{F}_{t_i},\p)\to L^2(\Omega,\ca{F}_{t_i},\p)$ be defined as
		\begin{align*}
			f(\xi)(\omega)=\bb{E}_i\parent{\ug{i+1}(\xscheme{i+1})}(\omega) + \psi\parent{t_i,\xscheme{i}(\omega),\xi(\omega), \gb{Z}_{t_i}(\omega)}h.
		\end{align*}
		For all $\xi \in L^2(\Omega,\ca{F}_{t_i},\p)$ and $\omega\in\Omega$. This function is well-defined by the properties of $\psi$ and the approximators. Let $\xi,\overline \xi\in L^2(\Omega,\ca{F}_{t_i},\p)$, then $\p$ a.s $|\psi(\xi)-\psi(\overline\xi)| \le h |\xi-\overline\xi|$, therefore
		\begin{align*}
			\norm{\psi(\xi)-\psi(\overline\xi)}_{L^2(\Omega,\ca{F}_{t_i}, \p)} \le h \norm{\xi-\overline \xi}_{L^2(\Omega,\ca{F}_{t_i}, \p)}.
		\end{align*}
		Taking $h<1$, which is independent of $i$, we can see that this function is a contraction on $L^2(\Omega,\ca{F},\p)$, and therefore, by applying Banach's fixed point theorem, we conclude the first result of this lemma. By standard computations,
		\begin{align*}
			\bb{E}\norm{\gbt{Z}{i}}^2_V &= \bb{E}\norm{\frac{1}{h}\bb{E}_i \parent{\hat{u}_{i+1}(\xscheme{i+1})\Delta W_i}}^2\\
			& \le \frac{1}{h^2}\bb{E}\parent{\bb{E}_i \norm{\hat{u}_{i+1}(\xscheme{i+1})\Delta W_i}_V}^2\le \frac{1}{h}\bb{E}\parent{|\hat{u}_{i+1}(\xscheme{i+1})|^2\norm{\Delta W_i}_V^2}\\
			&\le \frac{1}{h}\sqrt{\bb{E}|\hat{u}_{i+1}(\xscheme{i+1})|^4}\sqrt{\bb{E}\norm{\Delta W_i}^4_V} < \infty,
		\end{align*}
		where we used the fact that $W_t\in L^4(\Omega,\ca{F},\p;V)$. The proof is completed.
	\end{proof}

We intent to write $\gb{Z}_{t_i}$ as the average of some other process on $[t_i,t_{i+1}]$, to be consistent with the overline notation this process has to be denoted as $\widehat{Z}_{t}$ for $t\in [t_i,t_{i+1}]$.
\begin{lemma}\label{lemma:z-gorro}
	There exists a $V_0$-valued process $(\widehat{Z}_t)_{t\in[t_i,t_{i+1}]}$, which can be seen as an element of $\mathscr{N}_W([t_i,t_{i+1}];L_2(V_0,\R))$, such that,
	\begin{align*}
		\gbt{Z}{i} = \frac{1}{h}\bb{E}_i\parent{\integral \widehat{Z}_s ds}\in L^2(\Omega,\ca{F}_{t_i},\p;Q^{1/2}V).
	\end{align*}	
\end{lemma}
\begin{proof}
	Consider $N_t = \bb{E}(\hat{u}_{i+1}(\xscheme{i+1})|\ca{F}_t)$ for $t\in [t_i,t_{i+1}]$, this process is a square integrable martingale because $\hat{u}_{i+1}(\xscheme{i+1})\in L^2(\Omega,\ca{F}_{t_{i+1}},\p)$. By the martingale representation theorem \ref{lemma:mg-representation} there exists $(\widehat{Z}_t)_{t\in[t_i,t_{i+1}]}\in \mathscr{N}_W([t_i,t_{i+1}];L_2(V_0,\R))$, which ensures the a.e. Bochner integrability of $(\widehat{Z}_t)_{t\in[t_i,t_{i+1}]}$, such that,
	\begin{align*}
		N_t = N_{t_i} + \int_{t_i}^t\prom{\widehat{Z}_s,\cdot}_0 dW_s.
	\end{align*}
	By taking $t=t_i$,
	\begin{align*}
		\ug{i+1}(\xscheme{i+1}) = \bb{E}_i (\ug{i+1}(\xscheme{i+1})) + \integral\bra{\widehat{Z}_s}{\cdot}_{0}\ dW_s.
	\end{align*}
	It follows that,
	\begin{align*}
		h\gb{Z}_{t_i} = \bb{E}_i(\bb{E}_i(\ug{i+1}(\xscheme{i+1}))\Delta W_i) + 	\bb{E}_i\parent{\integral \bra{\widehat{Z}_s}{\cdot}_{0}\ dW_s\parent{W_{t_{i+1}} - W_{t_i}}}.
	\end{align*}
	Note that we took the equation from $\R$ to $V$. We can make the following elimination, 
	\begin{align*}
		\bb{E}_i(\bb{E}_i(\ug{i+1}(\xscheme{i+1}))\Delta W_i) =  \bb{E}_i(\ug{i+1}(\xscheme{i+1})) \bb{E}_i\Delta W_i = 0,
	\end{align*}
	which yields,
	\begin{align*}
		h\gb{Z}_{t_i} =\bb{E}_i\Big( \integral \bra{\widehat{Z}_s}{\cdot}_{0}\ dW_s\parent{W_{t_{i+1}} - W_{t_i}}\Big ).
	\end{align*}
	Recall that the representation \eqref{eq:w-representation} allows us to write $W_{t_{i+1}} - W_{t_i} = \sum_{j=1}^{\infty}f_j \sqrt{\lambda_j} (\beta_j(t_{i+1})-\beta_j(t_i))$, where the series converges in $L^2(\Omega,\ca{F},\p;V)$. Therefore, we can take the summation out of $\bb{E}_i$,
	\begin{align*}
		h\gb{Z}_{t_i} =\sum_{j=1}^{\infty}f_j\sqrt{\lambda_j}\bb{E}_i\Big( \integral \bra{\widehat{Z}_s}{\cdot}_{0}\ dW_s\integral d \beta_j(s)\Big ).
	\end{align*}
	Using Lemma \ref{lemma:n-brownian-integral} and the same argument as before with the $L^2 (\Omega,\ca{F},\p)$ limit
	\begin{align*}
		\integral\prom{\widehat{Z}_s,\cdot}_0 dW_s =\underset{n\to\infty}{\lim} \integral \prom{\widehat{Z}_s,\cdot}_0 dW^n_s= \sum_{k=1}^{\infty}\integral \prom{\widehat{Z}_s,Q^{1/2}f_j}_0 d\beta_s^k,
	\end{align*}
	we get,
	\begin{align*}
		h\gb{Z}_{t_i} &= \sum_{j=1}^{\infty}\sum_{k=1}^{\infty}\lambda_j^{1/2} f_j {\bb{E}}_i\parent{\integral \prom{\widehat{Z}_s,Q^{1/2}f_k}_0 d\beta^k_s\integral d\beta^j_s } \\
		&= \sum_{j=1}^{\infty} \bb{E}_i\parent{\integral \prom{\widehat{Z}_s,Q^{1/2}f_j}_0 Q^{1/2}f_j ds}.
	\end{align*}
	Where we used conditional Ito isometry. Last step is proving the following limit in $L^2(\Omega,\ca{F}_{t_i},\p;V)$,
	\begin{align*}
		\underset{n\to\infty}{\lim} \sum_{j=1}^n \integral \prom{\widehat{Z}_s, Q^{1/2}f_j}_0 Q^{1/2}f_j ds = \integral \widehat{Z}_s ds.
	\end{align*}
	Indeed,
	\begin{align*}
		& \bb{E}\norm{\integral \widehat{Z}_s ds - \sum_{j=1}^n \integral \prom{\widehat{Z}_s, Q^{1/2}f_j}_0 Q^{1/2}f_j ds}^2_V\\
		&= \bb{E}\norm{\integral\sum_{j=n+1}^{\infty}  \prom{\widehat{Z}_s, Q^{1/2}f_j}_0 Q^{1/2}f_j ds}^2_V \le h \bb{E}\integral\norm{\sum_{j=n+1}^{\infty}  \prom{\widehat{Z}_s, Q^{1/2}f_j}_0 Q^{1/2}f_j}^2_V ds\\
		&=h\bb{E}\integral\sum_{j=n+1}^{\infty}  |\prom{\widehat{Z}_s, Q^{1/2}f_j}_0|^2\prom{Q^{1/2}} ds\le h\bb{E}\integral \norm{\widehat{Z}_s}^2_0 ds \left(\sum_{j=n+1}^{\infty} \lambda_j \right),
	\end{align*}
	which approaches to $0$ as $n\to\infty$ because of $Q$ been trace class.  
\end{proof}

Recall the uniform partition $\pi$ with step $h$ from Subsection \ref{sec:numerical-scheme} and that $\Delta W_i = W_{t_{i+1}} - W_{t_i}$. 

\begin{lemma}
	The following holds:
	\begin{align*}
		\bb{E}_i\norm{\Delta W_i}_V^2 = \text{tr}(Q)h.
	\end{align*}
\end{lemma}
\begin{proof}
	Consider the identity mapping $I_V\colon V\to V$. By Ito isometry, one has that
	\[
	\begin{aligned}
		\bb{E}\norm{\Delta W_i}^2_V = &~{} \bb{E}\norm{\integral I_VdW_s}_V^2 = \bb{E}\integral\norm{I_V}^2_{L_2(V_0,V)}ds \\
		=&~{} h \norm{I_V}^2_{L_2(V_0,V)} = h\sum_{k=1}^{\infty}\lambda_k = \text{tr}(Q)h.
	\end{aligned} 
	\]
\end{proof}
It is useful to state and prove our main result to consider the following definition:
\begin{definition}
	For $i\in\set{0,...,N-1}$ let $(M_s)_{s\in[0,T]}$ be an integrable process and $(L_i)_{i\in\set{0,...,N-1}}$ be a set of random variables, all random objects taking values in some Hilbert $K$. We define,
	\begin{align}\label{eq:e-def}
		e_i(M,L_0) &= \bb{E}\integral \norm{M_s-L_0}_K^2 ds\qquad \text{and}\qquad e(M,L) = \sum_{i=0}^{N-1}e_i(M,L_i).
	\end{align}
	Also,
	\begin{align}\label{eq:z-barra}
		\overline{Z}_{t_i} = \frac{1}{h}\bb{E}_i\integral Z_s ds\; \in\; L^2(\Omega,\ca{F}_{t_i},\p;Q^{1/2}V).
	\end{align}
	Let $\varepsilon_i^v$, $\varepsilon_i^z$ given by
	\begin{equation}\label{errores}
		\varepsilon_i^{v,\eta}:=\underset{\theta\in\Theta_{\eta}}{\inf}\ \bb{E}|v_i(\xscheme{i}) - u^{\theta}_i(\xscheme{i})|^2, \qquad \varepsilon_i^{z,\eta}:=\underset{\theta\in\Theta_{\eta}}{\inf}\ \bb{E}\norm{z(\xscheme{i}) - z^{\theta}_i(\xscheme{i})}_0^2.
	\end{equation}
	Finally, consider
	\begin{align}\label{errores_epsilons}
		\varepsilon^{v,\eta} = \sum_{i=0}^{N-1} \varepsilon_i^v, \qquad \varepsilon^{z,\eta} = \sum_{i=0}^{N-1} \varepsilon_i^z.
	\end{align} 
\end{definition}
Previous definitions are related to the error committed in our scheme. Given the previous notation, consider the following assumptions which depends on the behavior of solution $(Y,Z)$ to stochastic equation \eqref{eq:bpsde} and how good the assumed scheme $X^{\pi}$ is.   
\begin{assumptions}\label{assumption:stochastic-sol-regularity}
	Assume that the processes $(Y,Z)\in\mathscr{S}^2_T(\R)\times\mathscr{N}_{W}(0,T;L^0_2(V,\R))$ satisfy that there exist $C>0$ and a function $\rho\colon(0,\infty)\to(0,\infty)$ such that,
	\begin{align}
		e(X,X^{\pi}) + e(Y,(Y_t)_{t\in\pi}) + e(Z,(\overline{Z}_t)_{t\in\pi})\le \rho(h),
	\end{align}
where $\rho(h)\rightarrow 0$ as $h\rightarrow  0$.
\end{assumptions}
This assumption holds in the finite dimensional case, where the control on regularity is precise and stipulated as a $\ca{O}(h)$. See e.g. \cite[Theorem $2.1$]{bruno}. Note that in general the distance used to measure the component related to $Y$ is always expressed in a $L^{\infty}$-type of distance. Meanwhile, terms related to $Z$ are measured in $L^2$-type of measure.   

\section{Universal Approximation Theorems and Deep-H-Onets}\label{sec:Universal Approximation Theorems and Deep-H-Onets}

In this section, our main objective will be to obtain precise bounds on the terms $\varepsilon_i^{v,\eta}$, $\varepsilon_i^{z,\eta}$ in \eqref{errores}. These bounds will be given in terms of infinite dimensional neural networks. Our main result for this section, Theorem \ref{theorem:infiniteapprox}, will provide the required control. First, we review some notation concerning finite dimensional NNs, we follow a slightly different notation of that given in \cite{yo}.

\subsection{Finite Dimensional Neural Networks}\label{sec:finite-dim-nn}
The NNs mathematical framework presented here is inspired by \cite{jentzen-nn-thoery}, we give a slightly simpler development that adapts to our motivations. Finite dimensional Neural Networks are building blocks to their infinite dimensional version, which we refer as Infinite Dimensional NN ($\text{NN}^{\infty}$ for short), and are also used as an intermediate step in the proof of the Universal Approximation theorem for $\text{NN}^{\infty}$. To fix ideas, in this section we focus on a setting where the input and output variables belong to multidimensional real spaces $\R^d$ and $\R^m$ respectively with $d,m\in\bb{N}$. The following definition introduce the notion of finite dimensional Neural Network with an arbitrary activation function.

\begin{definition}\label{def:NN-def}
	Consider $L+1\in\bb{N}$ as the number of layers within the network with $l_i\in\bb{N}$ neurons each for $i\in\set{0,...,L}$ where $l_0=d$ and $l_L=m$, weight matrices $\set{W_i\in\bb{R}^{l_i\times l_{i-1}}}_{i=1}^{L}$, bias vectors $\set{b_i\in \bb{R}^{l_i}}_{i=1}^{L}$, and an activation function $\sigma:\R\to\R$. Let $\theta=\set{W_i,b_i}_{i=1}^{L}$, which can be seen as an element of $\bb{R}^\kappa$ with $\kappa=\sum_{i=1}^{L}(l_i l_{i-1} + l_i)$, and a function $\sigma:\R\to\R$. We define the neural network $f^{\theta,\sigma}:\R^{l_0}\to\R^{l_L}$ as the following composition,
	\begin{align*}
		f^{\theta,\sigma}(x)=\parent{ A_L\circ\sigma\circ A_{L-1}\circ\cdots \circ A_2 \circ \sigma\circ A_1}(x),
	\end{align*}
	where $A_i:\R^{l_{i-1}}\to\R^{l_i}$ is an affine linear function such that $A_i(x)=W_ix+b_i$ for $i\in\set{1,...,L}$ and $\sigma$ is applied component-wise. One says that the function $f^{\theta,\sigma}$ is the realization of the parameter $\theta$ as a NN. Numbers $(l_i)_{i\in\set{0,...,L}}$ represents the amount of \textit{units} on each layer, note that the first layer has $l_0=d$ units and the last one has $l_L=m$ as they stand for the input and output variables respectively, the remaining $L-1$ layers are also known as hidden layers.
\end{definition}

We introduce some necessary conditions concerning activation functions. We follow the definitions given in \cite{chen95}.

\begin{definition}\label{def:TW-def}
	A function $\sigma:\R\to\R$ is called TW (Tauber-Wiener) if the set
	\begin{align*}
		\left\langle \set{\sum_{i=1}^N c_i\sigma(\lambda_i x + \theta_i)\Big| \lambda_i,\theta_i,c_i \in\R\ i\in\set{1,...,N}}\right\rangle
	\end{align*}
	is dense in $C([a,b])$ for $a,b\in\R$ and $a<b$.
\end{definition}
From the definition  it is not obvious how to determine if a function is TW, Chen and Chen \cite[Theorem $1$]{chen95} provide us with a result that makes it easier to know.
\begin{theorem}\label{theorem:tw}
	Suppose that $\sigma$ is a continuous function and that $\sigma\in S' (\R)$, the set of tempered distribution. Then, $\sigma$ is TW if and only if $\sigma$ is not a polynomial.
\end{theorem}
In this paper we work with an activation function known as ReLu denoted by $\sigma_{\text{ReLu}}\colon\R\to\R$ and is such that $\sigma_{\text{ReLu}}(x)=\max(x,0)$ for all $x\in\R$. We can see that this function satisfies hypothesis of Theorem \ref{theorem:tw}. In the following we make a formal definition of neural network and the set of parameters that defines them.

%\begin{align*}
%\int |\sigma_R \phi| &=\int_{[-R,R]} |x| |\phi| + |R|\int_{[-R,R]^{c}} |\phi| \le R\int_{[-R,R]}  |\phi| + R\int_{[-R,R]^c} \frac{|x|^2}{|x|^2}|\phi|\\
%&\le \norm{\phi}_{S'(\R)}\parent{2R^2 + R\int_{[-R,R]} \frac{1}{|x|^2}}.   
%\end{align*}
%Which proves that every ReLu activation function is TW. Now we define finite dimensional NNs.

\begin{definition}\label{def:finite-dim-nn}
	The set of parameters of Neural Networks associated to $l_0=d, l_L=m\in\bb{N}$ and a function $\sigma:\R\to\R$ is defined by,
	\begin{align*}
		\ca{N}_{\sigma,L,d,m} = \bigcup_{\kappa\in\bb{N}}\ca{N}_{\sigma,L,d,m,\kappa}
	\end{align*}
	where,
	\begin{align*}
		\ca{N}_{\sigma,L,d,m,\kappa} =  \Big\{ ~ \theta\in\R^{\kappa} ~ \Big| ~ &\theta=\set{W_i,b_i}_{i=1}^{L}, ~ l_0=d, ~ l_L=m, ~ W_i\in\R^{l_i\times l_{i-1}}, ~ b_i\in\R^{l_i},~ l_i\in\bb{N},\\
		&i\in\set{1,...,L},~\kappa= \sum_{i=1}^L (l_i l_{i-1} + l_i) \Big\}.
	\end{align*}
	Naturally,
	\begin{align*}
		\ca{N}_{\sigma,d,m,\kappa} = \bigcup_{L\in\bb{N}}\ca{N}_{\sigma,L,d,m,\kappa}\quad\text{and}\quad\ca{N}_{\sigma,d,m} = \bigcup_{L\in\bb{N}}\bigcup_{\kappa\in\bb{N}}\ca{N}_{\sigma,L,d,m,\kappa}
	\end{align*}
	Note that a parameter is eliminated when the union is taken over that parameter. For a set of parameters $\ca{N}\in\set{\ca{N}_{\sigma,d,m},\ca{N}_{\sigma,L,d,m},\ca{N}_{\sigma,d,m,\kappa}}$, the set of Neural Networks is then defined by,
	\begin{align*}
		\ca{R}(\ca{N}) = \set{f^{\theta,\sigma}\ \Big| \theta\in \ca{N}}.
	\end{align*}
	Here $f^{\theta,\sigma}:\Rd\to\R^m$.  
\end{definition}
Now, for completeness,we present two basic but important results. The first shows that NNs have a growth that is controlled by its parameters and activation function and the second that the composition of two NNs produce another NN bellowing to certain space $\ca{N}_{\sigma,L,d,m}$.
\begin{lemma}\label{lemma:nn-bound}
	Assume that $|\sigma(x)|\le|x|$ for any $x\in\R$. Let $\theta\in\ca{N}_{\sigma,2,d,m}$ such that $\theta=\set{W_1,b_1,W_2,b_2}$. Then there exist positive constants $c_1,c_2$, depending on $\theta$, such that,
	\begin{align*}
		\norm{f^{\theta,\sigma}(x)}^2 \le c_1\norm{x}^2 + c_2,\ \forall x\in\Rd.
	\end{align*}
	\begin{proof}
		Let $A\in\R^{m\times n},$ here we denote $\norm{A}^2=\sum_{i=1,j=1}^{m,n}A_{i,j}^2$, the Frobenius matrix norm. First we note that the function $f^{\theta,\sigma}$ takes the following form
		\begin{align*}
			f^{\theta,\sigma}(x) = \parent{\sum_{i=1}^n W_{2,k,i}\sigma\Big(\sum_{j=1}^d W_{1,i,j}x_j + b_{1,i}\Big) + b_{2,k}}_{k=1}^m.
		\end{align*}
		By a series of elemental computation and the application of Cauchy-Schwart inequality, we get,
		\begin{align*}
			\norm{f^{\theta,\sigma}(x)}^2 \le 4\norm{x}^2\norm{W_2}^2\norm{W_1}^2 + 4 \norm{W_2}^2\norm{b_1}^2 + 2 \norm{b_2}^2.
		\end{align*}
		Defining $c_1=4\norm{W_2}^2\norm{W_1}^2$ and $c_2=4 \norm{W_2}^2\norm{b_1}^2 + 2 \norm{b_2}^2$, we establish the required bound.
	\end{proof}
\end{lemma}
From the last lemma is straightforward that for $p\ge 2$,  $\norm{f^{\theta,\sigma}(x)}^p \le c_1\norm{x}^p + c_2$ for any $x\in\Rd$.
\begin{lemma}\label{lemma:nn-composition}
	Let $f^{\gamma}\in\ca{R}(\ca{N}_{\sigma,M,m,n})$ and $f^{\theta}\in\ca{R}(\ca{N}_{\sigma,L,d,m})$, then $f^{\gamma}\circ f^{\theta}\in\ca{R}(\ca{N}_{\sigma,L+M,d,n})$.
	\begin{proof}
		Let,
		\begin{align*}
			f^{\gamma} = B_M\circ\sigma\cdots\sigma\circ B_1\\
			f^{\theta} = A_L\circ\sigma\cdots\sigma\circ A_1.
		\end{align*}
		Then,
		\begin{align*}
			f^{\gamma}\circ f^{\theta} = B_M\circ\sigma\cdots\sigma\circ B_1\circ A_L\circ\sigma\cdots\sigma\circ A_1.
		\end{align*}
		Therefore the composition produce an additive property on the number of layers and $f^{\gamma}\circ f^{\theta}\in\ca{R}(\ca{N}_{\sigma,L+M,d,n})$.
	\end{proof}
\end{lemma}  
Previous lemma hints that the composition of NNs translate as a concatenation operation for its parameters, we introduce this notion in Definition \ref{def:param-concatenation}:
\begin{definition}\label{def:param-concatenation}
	For $\sigma,d,m$ we define the concatenation of parameters $\circ\colon\ca{N}_{\sigma,M,m,n}\times\ca{N}_{\sigma,L,d,m}\to\ca{N}_{\sigma,L+M,d,n}$ as,
	\begin{align}\label{eq:composition}
		\set{V_i,c_i}_{i=1}^M\circ\set{W_i,b_i}_{i=1}^L = \set{W_1,b_1,...,W_L,b_L,V_1,c_1,\dots,V_M,c_M}.
	\end{align}
	Then we have that for $\theta\in\ca{N}_{\sigma,L,d,m}$ and $\gamma\in\ca{N}_{\sigma,M,m,n}$ $f^{\theta}\circ f^{\gamma} = f^{\theta\circ\gamma}$.
\end{definition}
\begin{remark}
	Note that the order of composition at the left side of equation \eqref{eq:composition} differs from that of the right side. This is because the composition of functions is written in the opposite direction to the flow in a neural network (left to right).
\end{remark}
If the activation function $\sigma$ is continuous, the elements in $\ca{R}(\ca{N}_{\sigma,d,m})$ are continuous functions bellowing to $C(\Rd;\R^m)$. This is because they are composition of continuous mappings itself. Definition \ref{def:finite-dim-nn} is general, the first approximation theorem presented here is written a subset $\ca{H}$ of $\ca{N}_{\sigma,2,d,1}$ defined by
\begin{align}\label{eq:H}
	\ca{H}=\ca{N}_{\sigma,2,d,1}\cap\set{\theta\in\ca{N}_{\sigma,2,d,1} ~ | ~ \theta=\set{W_1,b_1,W_2}\in\R^{nd+n+n},\ b_2=0,\ n\in\bb{N}}.
\end{align}
Note that in this definition the free parameter $\kappa$ from definition \ref{def:finite-dim-nn} depends on the size $n\in\bb{N}$ of the first (and only) hidden layer in the following way, $\kappa=\sum_{i=1}^2 (l_i l_{i-1} + l_i) = nd + n + n + 1$. It is straightforward that a function $f^{\theta,\sigma}\in\ca{R}(\ca{H})$, set of real-valued mappings, takes the following form
\begin{align*}
	f^{\theta,\sigma}(x) = W_2\cdot\sigma(W_1 x + b_1)=\sum_{i=1}^n W_{2,i} \sigma\parent{\sum_{j=1}^{d} W_{2,i,j} x_j + b_{1,i} },
\end{align*}
for $\theta=\set{W_1,b_1,W_2}\in\R^{nd+n+n}$, $n\in\bb{N}$ and $x\in\Rd$. Now we state the first universal approximation theorem of this paper, it is proven by Chen and Chen in \cite[Theorem $3$]{chen95} and we present it here using our notation.
\begin{theorem}\label{thm:finite-dim-uat}
	Let $K$ be a compact set in $\Rd$, $U$ a compact set in $C(K)$ and $\sigma\colon\R\to\R$ a TW activation function. Then, for all $\varepsilon>0$ there exists a parameter $\theta$ depending on $g\in U$ as $\theta(g) = \set{W_1,b_1,W_2 (g)}\in\ca{H}$ such that
	\begin{align*}
		\underset{x\in K,g\in U}{\sup} |g(x) - f^{\theta(g)}(x)| < \varepsilon.
	\end{align*}
\end{theorem}
In particular, the latter theorem states that $\ca{R}(\ca{H})$ is dense in $C(K)$ endowed with the uniform topology in the sense that for every $\varepsilon$ there exist a NN with a sufficiently large hidden layer that meets the said accuracy in uniform distance. The following lemma extends Theorem \ref{thm:finite-dim-uat} proving the density of $\ca{R}(\ca{N}_{\sigma,2,d,m})$ in $C(K,\R^m)$ for a compact $K\subset\Rd$ and $m\ge 1$.
\begin{lemma}\label{thm:finite-dim-uat-m}
	Let $m\in\bb{N}$ with $m\ge 1$ and $K$ a compact set in $\R^d$. If the activation function $\sigma$ is TW, then $\ca{R}(\ca{N}_{\sigma,2,d,m})$ is dense in $C(K,\R^m)$.
	
	\begin{proof}
		Given $\varepsilon >0$ and a function $h=(h_1,...,h_m)\in C(K;\R^m)$ we need to find $f^{\theta,\sigma}=(f_1,...,f_m)\in\ca{R}(\ca{N}_{\sigma,2,d,m})$ such that
		\begin{align*}
			\underset{x\in K}{\sup}\norm{h(x) - f^{\theta,\sigma}(x)} < \varepsilon.
		\end{align*}
		First, observe that $\ca{R}(\ca{H})\subset \ca{R}(\ca{N}_{\sigma,2,d,1})$ which implies, by using Theorem \ref{thm:finite-dim-uat}, that $\ca{R}(\ca{N}_{\sigma,2,d,1})$ is also dense in $C(K)$ and therefore for every $i\in\set{1,...,m}$ we can find $f^{\theta^i,\sigma}$ with $\theta^i=\set{W_1^{i},b_1^i,W_2^i,b_2^i}$ and $\kappa^i=n^i d + n^i + n^i + 1$, depending on $\varepsilon$, such that
		\begin{align*}
			\underset{x\in K}{\sup} |h_i(x) - f^{\theta^i,\sigma}(x)| < \frac{\varepsilon}{\sqrt{m}}.
		\end{align*}
		Consider $\widehat{\theta}\in \ca{N}_{\sigma,2,d,m}$ with $\widehat{\theta}=\set{\widehat{W}_1, \widehat{b}_1, \widehat{W}_2, \widehat{b}_2}$ defined by
		\begin{align*}
			&\widehat{W}_1 = \begin{pmatrix}
				W_1^1\\
				\vdots\\
				W_1^m
			\end{pmatrix}\in \R^{\parent{\sum_{i=1}^m n^i} \times d},\ \  
			\widehat{b}_1 = \begin{pmatrix}
				b_1^1\\
				\vdots\\
				b_1^m
			\end{pmatrix}\in \R^{\sum_{i=1}^m n^i}\\
			&\widehat{W}_2 = \begin{pmatrix}
				W_2^{1,T} & 0      & 0\\
				0     & \ddots & 0 \\
				0     & 0      & W_2^{m,T}
			\end{pmatrix}\in \R^{m\times\sum_{i=1}^m n^i},\ \ 
			\widehat{b}_2 = \begin{pmatrix}
				b_2^1\\
				\vdots\\
				b_2^m
			\end{pmatrix}\in\R^m,
		\end{align*}
		and which satisfies that for $x\in\Rd$
		\begin{align*}
			f^{\widehat{\theta},\sigma}(x) &= \widehat{W}_2 \sigma(\widehat{W}_1 x + \widehat{b}_1) + \widehat{b}_2 = \begin{pmatrix}
				W^{1,T}_2\sigma(W_1^1 x + b_1^1) + b_2^1\\
				\vdots\\
				W^{m,T}_2\sigma(W_1^m x + b_1^m) + b_2^m
			\end{pmatrix} = \begin{pmatrix}
				f^{\theta^1,\sigma}(x)\\
				\vdots\\
				f^{\theta^m,\sigma}(x)
			\end{pmatrix}.
		\end{align*}
		Therefore, 
		\begin{align*}
			\underset{x\in K}{\sup} \norm{h(x) - f^{\widehat{\theta},\sigma}(x)} = \underset{x\in K}{\sup} \left( \sum_{i=1}^m |h_i(x) - f^{\theta_i,\sigma}(x)|^2 \right)^{1/2}< \varepsilon.
		\end{align*}
		This ends the proof.
	\end{proof}
\end{lemma}
The following lemma will be useful in the section devoted to $\text{NN}^{\infty}$, it is presented in \cite[Lemma C$.1$]{error-estimates-DOnets} as the Clipping lemma. Here we follow their proof as we need the explicit form of the NN given in the lemma. 
\begin{lemma}\label{lemma:clipping-lemma}	
	Let $\varepsilon>0$, $d\in\bb{N}$ and fix $0<R_1<R_2$. There exist a ReLu NN parameter $\theta\in\ca{N}_{\sigma_{\text{ReLu}},5,d,d}$, depending on $\varepsilon$ and $R_1$, such that
	\begin{align*}
		\begin{cases}
			\norm{f^{\theta}(x) - x} < \varepsilon, \quad \norm{x}\le R_1,\\
			\norm{f^{\theta}(x)} < R_2,\quad \forall x\in\Rd.
		\end{cases}
	\end{align*}
\end{lemma}

\begin{remark}
	The previous lemma is used in the proof of more general universal approximation theorems (See the following section), therefore it force us to stick to ReLu NNs from now on.
\end{remark}

\begin{proof}
		For any $a\in\R$, $\vec{a}$ represents the vector $(a,\dots,a)\in\Rd$ and as we are only working with ReLu activation function, we drop the $\sigma_{\text{ReLu}}$ from the NNs notation. Without loss of generality we may assume $\varepsilon < R_2-R_1$. Consider $\gamma\colon\Rd\to[-R_1,R_1]^d$ defined for $x\in\Rd$ as $\gamma(x)=\min(\max(x,-R_1),R_1)$, which depends on $R_1$ and can be represented exactly by a ReLu NN in $\ca{N}_{\sigma_{\text{ReLu}},3,d,d}$ as, 
		\begin{align*}
			\gamma(x) =-\max \left( -\max \left(x+\vec{R_1},0 \right) + 2\vec{R_1}, 0 \right) + \vec{R_1}. 
		\end{align*}
		Taking $\theta_{\gamma} = \set{I_d,\vec{R_1},-I_d,2\vec{R_1},-I,\vec{R_1}}
		$ follows that $\gamma = f^{\theta_{\gamma}}$. Note that for any $x\in [-R_1,R_1]^d$, $f^{\theta_{\gamma}}(x) = x$. The next step is to define a continuous function $\phi\colon\Rd\to\Rd$ by,
		\begin{align*}
			\phi(x)=
			\begin{cases}
				x,\ \norm{x}\le R_1\\
				R_1 \frac{x}{\norm{x}},\ \norm{x} > R_1.
			\end{cases}
		\end{align*}
		We have that $\phi \in C([-R_1,R_1]^d)$, then, by Theorem \ref{thm:finite-dim-uat-m}, there exists $f^{\theta_{\varepsilon}}\in\ca{N}_{\sigma_{\text{ReLu}},2,d,d}$ such that,
		\begin{align*}
			\underset{x\in[-R_1,R_1]^d}{\sup} \norm{\phi(x) - f^{\theta_{\varepsilon}}(x)} < \varepsilon.
		\end{align*} 
		Define now $\theta =\theta_{\varepsilon}\circ \theta_{R_1}$, which is well defined and belong to $\ca{N}_{\sigma_{\text{ReLu}},5,d,d}$ by Lemma \ref{lemma:nn-composition} and Definition \ref{def:param-concatenation}. Then, for any $\norm{x}\le R_1$.
		\begin{align*}
			\norm{f^{\theta}(x) - x} = \norm{f^{\theta_{\varepsilon}}(f^{\theta_{\gamma}}(x)) - \phi(x)} = \norm{f^{\theta_{\varepsilon}}(x) - \phi(x)} < \varepsilon,
		\end{align*}
		and,
		\begin{align*}
			\underset{x\in\Rd}{\sup}\norm{f^{\theta}(x)} \le \underset{x\in [-R_1,R_1]^d}{\sup}\norm{f^{\theta_{\varepsilon}}(x) - \phi(x)} + R_1 < R_2. 
		\end{align*}
		This finishes the proof.
	\end{proof}

\subsection{Infinite Dimensional Neural Networks: Hilbert-valued DeepOnets}\label{sec:infinite-dim-nn}

In this section we work with a particular type of $\text{NN}^{\infty}$ called DeepOnets. Based on the definitions given in \cite{error-estimates-DOnets}, we provide a proper and rigorous treatment of this object and prove important results that allows them to be used on our PDE and stochastic setting.

\medskip

Through this entire section $(H,\prom{\cdot,\cdot}_H,\norm{\cdot}_H)$ and $(W,\prom{\cdot,\cdot}_W,\norm{\cdot}_W)$ will denote any Hilbert space with orthonormal basis $(e_i)_{i\in\bb{N}}$ and $(g_i)_{i\in\bb{N}}$ respectively, $H$ is equipped with a Borel probability measure $\mu$. In the following we are devoted to study the approximation of functionals of the form $F\colon H\to W$ by functions parameterized by finite dimensional parameters. The main idea to define such functions is to take a sufficiently large $d\in\bb{N}$ such that the approximations $\sum_{i=1}^d\prom{x,e_i}_H e_i$ are good enough to approximate $x\in H$ and encode $x$ as the vector $(\prom{x,e_1}_H,...,\prom{x,e_n}_H)\in\R^d$, then use a finite dimensional neural network to go from $\Rd$ to $\R^m$ for some $m\in\bb{N}$. At last, we take the resulting vector to $W$ by considering its $m$ components as coefficients for $\set{g_1,...,g_m}$. The structure of Hilbert spaces allow us to take advantage of results such as Lemma \ref{lemma:hilbert_aris}, which we present below with a proof due to Aris Daniilidis. Note that it is valid for every Hilbert space.

\begin{lemma}[Daniilidis]\label{lemma:hilbert_aris}
	Let $K$ be a compact set on $H$. For every $k\in\bb{N}$ consider the operator $P_k:H\to H$ defined as $P_k(x) = \sum_{i=1}^k \bra{x}{e_i}_H e_i$ for $x\in H$. Then, for every $\varepsilon>0$ there exists $k\in\mathbb{N}$ such that for all $x\in K$,
	\begin{align*}
		\norm{P_kx - x}_H \le \varepsilon. 
	\end{align*}
\end{lemma}
\begin{proof}
	First, lets establish that for all $k\in\mathbb{N}$, $P_k\in L(H)$ and $\norm{P_k}_H\le 1$. $P_k$ is clearly linear, to prove the bound let $x$ be any non-zero vector in $H$,
	\begin{align*}
		\norm{P_k x}_H^2 = \norm{\sum_{i=1}^k \bra{x}{e_i}_He_i }_H^2 = \sum_{i=1}^k |\bra{x}{e_i}_H|^2 \le \sum_{i=1}^{\infty} |\bra{x}{e_i}_H|^2 = \norm{x}_H^2.
	\end{align*}
	This means that $\norm{P_k}_{L(H)}\le 1$.\\
	
	We argue by contradiction. Suppose that there exists $\varepsilon >0$ such that for all $n\in\mathbb{N}$ we can find $x_n\in K$ verifying $\norm{P_n(x_n) - x_n}_H\ge \varepsilon$. Due to the compactness of $K$, there is a subsequence that converges to some $x\in H$, we denote this subsequence as $x_n$ as well. Then,
	\begin{align*}
		\norm{P_n(x_n) - x_n}_H&\le \norm{P_n(x_n) - P_n(x)}_H + \norm{P_n(x) - x}_H + \norm{x - x_n}_H\\
		&\le 2\norm{x_n - x}_H + \norm{P_n(x) - x}_H.
	\end{align*}
	The first term can be made as small as we want due to the convergence of $x_n$ to $x$ and the second because we have that $P_n(x)\to x$ in $H$ as $n\to\infty$. Then, for some large $n$ we can break the bound and thus, the contradiction.
\end{proof}
From now on we fix $\sigma=\sigma_{\text{ReLu}}$. 
\begin{definition}\label{def:deeponets}
	Recall Definition \ref{def:finite-dim-nn}. Given $L,d,m\in\bb{N}$ consider the functions
		
	\noindent
	\begin{minipage}{.45\linewidth}
		\begin{align*}
			\ca{E}_{H,d}: H&\longrightarrow \Rd\\
			x &\longmapsto  \bigg(\prom{x,e_i}_H\bigg)_{i=1}^d,
		\end{align*}
	\end{minipage}%
	\begin{minipage}{.45\linewidth}
		\begin{align*}
			\widehat{\ca{E}}_{W,m}: \R^m&\longrightarrow W\\
			a &\longmapsto \sum_{i=1}^m a_i g_i.
		\end{align*}
	\end{minipage}

	\medskip

	Let $\theta\in\ca{N}_{\sigma,L,d,m}$, for $(H,d,\theta,m,W)$ we define the DeepOnet $F^{H,d,\theta,m,W}:H\to W$ by
	\begin{align}\label{eq:DO-def}
		F^{H,d,\theta,m,W} = \widehat{\ca{E}}_{W,m}\circ f^{\theta}\circ \mathcal{E}_{H,d}.
	\end{align}
	Unless is extremely necessary, we omit $H,W$ and just use $F^{d,\theta,m}$. Also, define the following sets of DeepOnets parameters,
	\begin{align*}
		\ca{N}^{H\to W}_{\sigma} &= \bigcup_{d,m\in\bb{N}} \set{d}\times\ca{N}_{\sigma,d,m}\times\set{m},\\
		\ca{N}^{H\to W}_{\sigma,L} &= \bigcup_{d,m\in\bb{N}}\set{d}\times\ca{N}_{\sigma,L,d,m}\times\set{m}.
	\end{align*}
	With $L\in\bb{N}$, observe that $\ca{N}^{H\to W}_{\sigma,L}\subset\ca{N}^{H\to W}_{\sigma}$ (the less parameters specified, the bigger the set). Let $\ca{N}=\ca{N}^{H\to W}_{\sigma}$ or $\ca{N}=\ca{N}^{H\to W}_{\sigma,L}$, it is straightforward to define,
	\begin{align*}
		\ca{R}(\ca{N}) = \set{F^{H,d,\theta,m,W}\ \Big|\ (d,\theta,m)\in\ca{N}}.
	\end{align*} 
\end{definition}
Note that $d$ is not readable as an input dimension, here it becomes a parameter of the DeepOnet and represents how many elements of the base $(e_i)_{i\in\bb{N}}$ we are using to project with in order to get the finite dimensional representation $(\prom{x,e_i}_H)_{i=1}^d$ for $x\in H$. Last action is carried out by mapping $\ca{E}_{H,d}$. The same goes for $m$ but in the opposite direction and in this case, it is done by $\widehat{\ca{E}}_{W,m}$, which allows us to take a collection of real numbers to a Hilbert space. Observe that functions in $\ca{R}(\ca{N})$ are continuous because they are composition of continuous functions itself.
\begin{remark}\label{remark:DO-def}
	We remark the following,
	
	\begin{itemize}
		\item We have that $\ca{E}_{\Rd,d}=I_d$ and $\widehat{\ca{E}}_{\Rd,d}=I_d$. Note that with this consideration we recover the finite dimensional theory by taking $H=\Rd$ and $W=\R^m$.
		
		\item We could just denote $F^{d,\theta,m}$ as $F^{\theta}$ because the information about the input and output dimension of the NN is codified in the parameter $\theta$, but we decide to specify $d,m$ for a better understanding. Also the order of the parameters makes clearer in which order the composition are taken.
		
		\item Note that the number of parameters to define a DeepOnet is the same as of NNs only adding $d,m$.
		
		\item If $H$ is a functional space such as $L^2(\Rd,\borel,dx)$, DeepOnets also admits a \comillas{neural network representation} where the first layer is in some sense dense as has an infinite number of units which are all captured by $\prom{\cdot,\cdot}$ to be transferred to the next finite layer. 
	\end{itemize}
\end{remark}
%\begin{lemma}\label{lemma:DO-composition-stability}
%	The set $\ca{R}(\ca{N}^{\infty}_{\sigma})$ is stable under left composition with elements on $\ca{R}(\ca{N}_{\sigma,1,1})$.
%\end{lemma}
%\begin{proof}
%	The result follows from Lemma \ref{lemma:nn-composition} together with definitions \ref{def:deeponets} and \ref{def:param-concatenation}. Indeed, let $(\theta,d)\in\ca{N}^{\infty}_{\sigma}$ and $\gamma\in\ca{N}_{\gamma,1,1}$ then,
%	\begin{align*}
%		f^{\gamma}\circ F^{\theta,d} =  f^{\gamma}\circ f^{\theta}\circ\ca{E}^d = f^{\gamma\circ\theta}\circ\ca{E}^d=F^{\gamma\circ\theta,d}.
%	\end{align*}
%	With $(\gamma\circ\theta,d)\in\ca{N}^{\infty}_{\sigma}$ and the dimensions fits because the input and output dimension of $f^{\gamma}$ are $1$. Thus, the lemma follows.
%\end{proof}
\begin{prop}[See e.g. Theorem $4$ in Chen and Chen \cite{chen95}]\label{prop:functional-sup-uat}
	Let $m\in\bb{N}$, $K\subset H$ be a compact set and $f:K\to \R^m$ be a continuous function. Then, for any $\varepsilon>0$ there exists $(d,\theta,m)\in\ca{N}^{H\to\R^m}_{\sigma,2}$ such that,
	\begin{align*}
		\underset{x\in K}{\sup} \norm{F^{d,\theta,m}(x)-f(x)} \le \varepsilon.
	\end{align*}
	In other words, $\set{F|_K\colon F\in\ca{R}(\ca{N}^{\infty}_{\sigma,2})}$ is dense in $C(K)$ endowed with the uniform norm.
\end{prop}
\begin{proof}
	Consider the operators $P_k$ from Lemma \ref{lemma:hilbert_aris}. The said Lemma tells us that for $\delta_k\searrow0$ we can find a set of natural numbers $(n_k:=n(\delta_k))_{k\in\bb{N}}$ such that,
	\begin{align*}
		\forall k\in\mathbb{N},\ \forall u\in K, \norm{P_{n_k}(u) - u}_H < \delta_k.
	\end{align*}
	Given the continuity of $P_k$, $P_k(K)$ is also a compact set in $H$ for all $k\in\mathbb{N}$. Now we prove that the set
	\begin{align*}
		A:=\parent{\bigcup_{k=1}^\infty P_{n_k}(K)}\cup V,
	\end{align*}
	is also compact in $H$. Indeed, let $(x_i)_{i\in\bb{N}}$ be a sequence in $A$. If there exists a subsequence such that it remains in $K$, there is nothing to prove because $K$ is compact. The other case is that we can extract an infinite subsequence that lies in the infinite union. This means that there exists $(k_i)_{i\in\bb{N}}\subset \mathbb{N}$ and $(u_i)_{i\in\bb{N}} \subset K$ such that,
	\begin{align*}
		x_i = \sum_{j=1}^{n_{k_i}} \bra{u_i}{e_j}e_j.
	\end{align*}
	Due to compactness of $K$, up to a subsequence that we also denote $(u_i)_{i\in\bb{N}}$ as well, $(u_i)_{i\in\bb{N}}$ converges to some $u\in K$. We have two options, the first is that the sequence $(k_i)_{i\in\bb{N}}$ does not grow to infinite when $i\nearrow \infty$ and thus, up to a subsequence on $i$, we can find $\iota$ such that $\forall i\ge \iota$, $k_i=k_{\iota}$ which implies that, for $i\ge \iota$,
	\begin{align*}
		x_i = \sum_{j=1}^{n_{k_{\iota}}} \bra{u_i}{e_j}_H e_j\longrightarrow\sum_{j=1}^{n_{k_{\iota}}} \bra{u}{e_j}_H e_j\in P_{n_{k_{\iota}}}\subset A.
	\end{align*}
	The second option is that up to a subsequence, $k_i\nearrow \infty$ as $i\nearrow\infty$, note that
	\begin{align*}
		x_i = \sum_{j=1}^{n_{k_i}} \bra{u_i}{e_j}_H e_j = P_{n_{k_{i}}} (u_i),
	\end{align*}
	and then,
	\begin{align*}
		\norm{x_i - u}_H \le \norm{P_{n_{k_{i}}} (u_i) - u_i}_H + \norm{u_i - u}_H\le \delta_{k_{i}} + \norm{u_i - u}_H,
	\end{align*}
	where, taking $i\to \infty$ we prove that, up to a subsequence, $x_i\longrightarrow u\in K\subset A$. Thus, $A$ is compact in $H$. \\
	
	The next step is to use the well-known Tietze-Urysohn theorem \cite[Chapter $4$, Theorem $35.1$]{munkres} which gives us a continuous extension $f_{\text{ex}}:A\to \R^m$ with $f_{\text{ex}}(x) = f(x)$ for $x\in K$. The compactness of $A$ implies that $f_{\text{ex}}$ is uniformly continuous, then, for $\varepsilon>0$ we can find $\delta>0$ depending only on $\varepsilon$ such that $\norm{x-y}_H<\delta$ implies $\norm{f_{\text{ex}}(x) - f_{\text{ex}}(y)}<\varepsilon$. Lets fix $k\in\mathbb{N}$ such that $\delta_k < \delta$, let $F\colon K\to\mathbb{R}^m$ be a function to be specified later and $x$ any element of $K$, then
	\begin{align*}
		\norm{f(x) - F(x)}\le \norm{f_{\text{ex}}(x) - f_{\text{ex}}(P_{n_k}(x))} + \norm{f_{\text{ex}}(P_{n_k}(x)) - F(x)} < \frac{\varepsilon}{2} + \norm{f_{\text{ex}}(P_{n_k}(x)) - F(x)}.
	\end{align*}
	By the continuity of $\ca{E}_{H,n_k}$ follows that $\ca{E}_{H,n_k}(K)$ is a compact set in $\R^{n_k}$. Consider the function $\bar{f}$ defined by
	\begin{align*}
		\bar{f}\colon \ca{E}_{H,n_k}(K)&\longrightarrow \R^m\\
		y\quad &\longmapsto \bar{f}(y)=f_{\text{ex}}\parent{\sum_{j=1}^{n_k} y_j e_j}.
	\end{align*} 
	Note that the extension is essential because $\ca{E}_{H,n_k}$ could not be a subset of $K$, where $f$ is defined. By the universal approximation Theorem \ref{thm:finite-dim-uat-m} there exists $\theta\in\ca{N}_{\sigma,2,n_k,m}$ such that
	\begin{align*}
		\underset{y\in \ca{E}_{H,n_k}(K)}{\sup}  \norm{\bar{f}(y) - f^{\theta}(y)} &= \underset{y\in \ca{E}_{H,n_k}(K)}{\sup}  \norm{f_{\text{ex}}\parent{\sum_{i=1}^{n_k} y_i e_i} - f^{\theta}(y)}\\
		&= \underset{x\in K}{\sup}  \norm{f_{\text{ex}}\parent{\sum_{i=1}^{n_k} \prom{x,e_i}_H e_i} - f^{\theta}\parent{\parent{\prom{x,e_i}_H}_{i=1}^{n_k}}}\\
		&= \underset{x\in K}{\sup} \norm{f_{\text{ex}}\parent{P_{n_k}(x)} - \parent{\widehat{\ca{E}}_{\R^m,m}\circ f^{\theta}\circ \ca{E}_{H,n_k}}(x) }	
		< \frac{\varepsilon}{2}.
	\end{align*}
	Recall the first point in Remark \ref{remark:DO-def}. It suffices to take $(n_k,\theta,m)\in\ca{N}^{W\to\R^m}_{\sigma,2}$ which concludes the proof.
\end{proof}
%Following proposition follows as an extension of Proposition \ref{prop:functional-sup-uat}.
%\begin{prop}\label{prop:chen-and-chen-rn}
%	Given $m\in\bb{N}$, a compact set $K\subset H$, a continuous function $f:K\to \R^m$ and $\varepsilon>0$, there exists $(\theta,d)\in\ca{N}^{\infty}_{\sigma,2,m}$ such that,
%	\begin{align*}
%		\underset{x\in K}{\sup} \norm{F^{\theta,d}(x)-f(x)}_{\R^m} \le \varepsilon.
%	\end{align*}
%	In other words, $\set{F|_K\colon F\in\ca{R}(\ca{N}^{\infty}_{\sigma,2,m})}$ is dense in $C(K;\R^m)$ endowed with the uniform norm.
%\end{prop}

The main result of this section, concerning the approximation of a square integrable functional is presented below and is closely related to the approximation of a solution to equation \eqref{eq:pde}. We divide the proof in steps for a clear reading and follow the lines of \cite[Theorem $3.1$]{error-estimates-DOnets}. 

\begin{theorem}\label{theorem:infiniteapprox}
	Let $(W,\prom{\cdot,\cdot}_W,\norm{\cdot}_W)$ be a separable Hilbert space with orthonormal basis $(g_i)_{i\in\bb{N}}$. Let $G\colon H\to W$ be a $L^2(H,\mu;W)$ mapping. Then, for any $\varepsilon>0$ there exist a DO $F^{d,\theta,m}:H\to W$ such that,
	\begin{align*}
		\int_{H}\norm{G(x)-F^{d,\theta,m}(x)}_W^2\mu(dx) \le \varepsilon.
	\end{align*} 	
\end{theorem}
\begin{proof}
	\textbf{Step 1.} Let $\varepsilon>0$ and define $\delta = \sqrt{\varepsilon / 8}$. First we prove that without loss of generality we can assume that $G$ is bounded. Consider $M>0$ and 
	\begin{align*}
		G_M (x) :=
		\begin{dcases}
			G(x),\ &\norm{G(x)}_W\le M\\
			M\frac{G(x)}{\norm{G(x)}_W},\ &\sim
		\end{dcases}
	\end{align*}
	Then, for any function $F\colon H\to W$ we get,
	\begin{align*}
		\norm{G-F}_{L^2(H,\mu;W)}\le\norm{G-G_M}_{L^2(H,\mu;W)} + \norm{G_M-F}_{L^2(H,\mu;W)}. 
	\end{align*}
	We have that $\norm{G_M - G}^2_W\to 0$ and $\norm{G_M - G}^2_W\le 4\norm{G}^2_W$ $\mu$-a.e., so applying dominate convergence theorem we take $M$ such that,
	\begin{align*}
		\norm{G-F}_{L^2(H,\mu;W)}\le\delta + \norm{G_M-F}_{L^2(H,\mu;W)}. 
	\end{align*}
	Then, assuming $\norm{G}_W\le M$ on $H$, we prove that $\norm{G-F}_{L^2(H,\mu;W)}<\delta$ for certain DeepOnet $F$.\\
	
	\textbf{Step 2.} By Lusin's (\cite{bogachev2}) theorem, there exists a compact set $K=K(\delta,M)\subset H$ such that $G|_{K}$ is continuous and $\mu\parent{H\setminus K}< \frac{\delta^2}{M^2}$. Now, consider the compact set $K'=G(K)\subset W$. In virtue of Lemma \ref{lemma:hilbert_aris}, there exist $\kappa=\kappa(K')\in\bb{N}$ such that,
	\begin{align*}
		\underset{w\in K'}{\sup} \norm{w - P_{\kappa}(w)}_W \le \delta.
	\end{align*}
	Let $\widetilde{G}=P_{\kappa}\circ G\colon K\to W$. Note that,\\
	\begin{align*}
		\underset{x\in K}{\sup}\norm{G(x) - \widetilde{G}(x)}_W = \underset{w\in K'}{\sup}\norm{w - P_{\kappa}(w)}_W \le \delta.
	\end{align*}
	
	\textbf{Step 3.} Applying Proposition \ref{prop:functional-sup-uat} for the continuous function $\ca{E}_{W,\kappa}\circ\widetilde{G}\colon K\to\R^{\kappa}$, we can take $(d,\theta_1,\kappa)\in\ca{N}^{H\to\R^{\kappa}}_{\sigma,2}$ such that,
	\begin{align*}
		\underset{x\in K}{\sup} \norm{ F^{H,d,\theta_1,\kappa,\R^{\kappa}}(x) - (\ca{E}_{W,\kappa}\circ \widetilde{G})(x)} < \delta.
	\end{align*}
	Take any $x\in K$ and the DO generated by $(H,d,\theta_1,\kappa,W)$,
	\begin{align}\label{eq:bound1_uat}
		\norm{F^{(H,d,\theta_1,\kappa,W)}(x) - \widetilde{G}(x)}_W &= \norm{(\widehat{\ca{E}}_{W,\kappa}\circ f^{\theta}\circ \ca{E}_{H,d})(x) - \widetilde{G}(x)}_W\nonumber\\
		& = \norm{\sum_{i=1}^{\kappa} (f^{\theta}\circ \ca{E}_{H,d})(x)_i g_i - \sum_{i=1}^{\kappa} \prom{G(x),g_i}_W g_i }_W\nonumber\\
		&= \norm{(f^{\theta}\circ \ca{E}_{H,d})(x) - \parent{\prom{G(x),g_i}_W}_{i=1}^{\kappa}}_{\R^{\kappa}}\nonumber\\
		&= \norm{F^{H,d,\theta_1,\kappa,\R^{\kappa}}(x) - (\ca{E}_{H,\kappa}\circ \widetilde{G})(x)}_{\R^{\kappa}} < \delta.
	\end{align}
	Then, by using previous estimate, Lemma \ref{lemma:hilbert_aris} and that $G$ is bounded, one has the following bound
	\begin{align*}
		\norm{F^{H,d,\theta_1,\kappa,W}(x)}_W \le \norm{F^{H,d,\theta_1,\kappa,W}(x) - \widetilde{G}(x)}_W + \norm{\widetilde{G}(x) - G(x)}_W + \norm{G(x)}_W < 2\delta + M.
	\end{align*}
	\textbf{Step 4.}
	Applying the clipping Lemma \ref{lemma:clipping-lemma} with $\delta$, $\kappa$, $R_1=M+2\delta$ and $R_2=2M$, note that we can assume $\delta$ small enough such that $R_1<R_2$, we can take $\theta_2\in\ca{N}_{\sigma,5,\kappa,\kappa}$ such that,
	\begin{align}\label{eq:clipping-gamma}
		\begin{dcases}
			\norm{f^{\theta_2}(x) - x} < \delta,\ &\norm{x} < M + 2\delta\\
			\norm{f^{\theta_2}(x)} \le 2M,\ &\forall x\in \bb{R}^{\kappa}.
		\end{dcases}
	\end{align}
	Recall that the norm used in previous equation is the usual norm in $\R^{\kappa}$ and that during this entire section, $\sigma=\sigma_{\text{ReLu}}$. Consider the following composition and its equivalences,
	\begin{align*}
		\widehat{\ca{E}}_{W,\kappa}\circ f^{\theta_2}\circ \widehat{\ca{E}}_{\R^{\kappa}}\circ f^{\theta_1}\circ \ca{E}_{H,d} = \widehat{\ca{E}}_{W,\kappa}\circ f^{\theta_2\circ\theta_1}\circ \ca{E}_{H,d}=F^{H,d,\theta_1\circ\theta_2,\kappa,W}.
	\end{align*}
	Where we made use of Definition \ref{def:param-concatenation}. Such DO satisfies the following,
	\begin{align*}
		\norm{F^{H,d,\theta_2\circ\theta_1,\kappa,W}(x) - \widetilde{G}(x)}_W &\le \norm{F^{H,d,\theta_2\circ\theta_1,\kappa,W}(x) - F^{H,d,\theta_1,\kappa,W}(x)}_W + \norm{F^{H,d,\theta_1,\kappa,W}(x) - \widetilde{G}(x)}_W\\
		&\le \norm{\sum_{i=1}^{\kappa}f^{\theta_2}_i\parent{f^{\theta_1}\parent{\ca{E}_{H,d}(x)}} g_i -  \sum_{i=1}^{\kappa} \parent{f^{\theta_1}\circ \ca{E}_{H,d}}_i (x) g_i}_W + \delta\\
		&\le\norm{f^{\theta_2}\parent{f^{\theta_1}\parent{\ca{E}_{H,d}(x)}} - f^{\theta_1}\parent{\ca{E}_{H,d}(x)}}_{\R^{\kappa}} + \delta < 2\delta,
	\end{align*}
	where we used estimates \eqref{eq:bound1_uat} and \eqref{eq:clipping-gamma}.\\
	
	\textbf{Step 5.} Now we use all previous bounds, let $F=F^{H,d,\theta_2\circ\theta_1,\kappa,W}$ with $(d,\theta_2\circ\theta_1,\kappa)\in\set{d}\times\ca{N}_{\sigma,7,d,\kappa}\times\set{\kappa}$, then 
	\begin{align*}
		\int_H \norm{G(x) - F(x)}_W^2 \mu (dx) &= \int_{H\setminus K} \norm{G(x) - F(x)}_W^2 \mu (dx) + \int_{K} \norm{G(x) - F(x)}_W^2 \mu (dx)\\
		&\le 2\int_{H\setminus K} \norm{G(x)}_W^2\mu(dx) + 2\int_{H\setminus K}\norm{F(x)}^2_W\mu(dx) \\
		&~{} \quad  + 2\int_{K} \norm{G(x) - \widetilde{G}(x)}_W^2\mu(dx) + 2\int_{K} \norm{\widetilde{G}(x) -  F(x)}_W^2\mu (dx)\\
		&\le \mu\parent{H\setminus K}(2M^2 + 2M^2) + 2\delta^2 + 2\delta^2 \le 8\delta^2=\varepsilon,
	\end{align*} 
	which is the desired conclusion.
\end{proof}
Note that the theorem above only contribute with the existence of a parameter $(d,\theta,m)$ such that the generated DO is a good approximation, in order to overcome the said \textit{curse of dimensionality} we may have to provide proper bounds on the size of $(d,\theta,m)$. Following lemma provides us with a useful bound for DeepOnets.

\begin{remark}\label{rem_5p4}
	Recall the notation from Step $5$ from the proof above. Given the parameters $(d,\theta_2\circ\theta_1,\kappa)\in\set{d}\times\ca{N}_{\sigma,7,d,\kappa}\times\set{\kappa}$, we have that $\theta_2\circ\theta_1\in\R^{\eta}$ for some $\eta\in\bb{N}$; therefore
	\begin{align}\label{eq:remark-bound}
		\underset{(p,\theta,q)\in\bb{N}\times\R^{\eta}\times\bb{N}}{\inf}\int_H \norm{G(x) - F^{p,\theta,q}(x)}_W^2 \mu (dx)\le \int_H \norm{G(x) - F^{d,\theta_2\circ\theta_1,\kappa}(x)}_W^2 \mu (dx)\le \varepsilon.
	\end{align}
	This observation allows us to state that for any $\varepsilon>0$ we can find a sufficiently large $\eta\in\bb{N}$ such that the left side of \eqref{eq:remark-bound} is bounded by $\varepsilon$.
\end{remark}

\begin{lemma}\label{lemma:DO-bound}
	Let $p\ge 2$ and $(d,\theta,m)\in\ca{N}^{H\to W}_{\sigma,2}$, then there exists $c_1,c_2>0$ such that $|F^{\theta,d}(x)|^p\le c_1\norm{x}^p_H + c_2$ for every $x\in H$.
\end{lemma} 

\begin{proof}
	Let $x\in H$, then by using Lemma \ref{lemma:nn-bound} there exists $a_1,a_2>0$ such that,
	
	Defining $c_1 = 2^{\frac{p-2}{2}}a_1^{p/2}$ and $c_2 = 2^{\frac{p-2}{2}}a_2^{p/2}$ concludes the proof.
\end{proof}

\section{Main Result}\label{sec:Main Result}

Now we are ready to state and prove the main result of this paper. Recall the properties of approximators in Subsection \ref{sec:numerical-scheme}.

\begin{theorem}\label{MT1}
	Under Assumptions \ref{assumptions}, \ref{assumptions3} and \ref{assumption:stochastic-sol-regularity}, there exists a constant $C>0$ independent of the partition such that for sufficiently small $h$,
	\begin{align}
		&\underset{i=0,...,N-1}{\max}\esp{Y_{t_i}-\hat{u}_{i}(\xscheme{i})}^2 + \sum_{i=0}^{N-1}\bb{E}\parent{\integral \norm{Z_t-\hat{z}_{i}(\xscheme{i})}_0^2  dt}\label{eq:MT-eq}\\
		&~{} \qquad \qquad\le
		C\Big[ h +\esp{\phi(X_T)-\phi(X_T^\pi)}^2+ N\varepsilon^{v,\eta} + \varepsilon^{z,\eta} +
		\rho(h)  \Big],
		\nonumber
	\end{align}
	with $\varepsilon^{v,\eta}$, $\varepsilon^{z,\eta}$ given in \eqref{errores_epsilons}.
\end{theorem}

\begin{proof}
	{\color{black} \textbf{Step 1:}} Recall $\widehat{\ca{V}}_{t_i}$ introduced in \eqref{eq:nu-gorro}. The purpose of this part is to obtain a suitable bound of the term  $\esp{Y_{t_i}-\widehat{\ca{V}}_{t_i}}^2$ in terms of more tractable terms. We have
	
	\begin{lemma}
		There exists $C>0$ fixed such that for any $0<h<1$ sufficiently small, one has
		\begin{align}\label{step1}
			\esp{Y_{t_i}-\widehat{\ca{V}}_{t_i}}^2 \le &~{} Ch^2+C\bb{E}\integral|Y_s-Y_{t_i}|^2ds+ C\bb{E}\integral\norm{Z_s-\overline{Z}_{t_i}}_V^2ds +Ch \bb{E}\integral \psi(\Theta_r)^2dr  \nonumber  \\
			&~{} + C(1+Ch)\bb{E} \left| Y_{t_{i+1}}-\ug{i+1} (X^\pi_{t_{i+1}}) \right|^2,
		\end{align}
		with $\Theta_r=(r,X_r,Y_r,Z_r)$.
	\end{lemma}

	The rest of this subsection is devoted to the proof of this result. 
	\begin{proof}

		Subtracting the equation \eqref{eq:bpsde} between $t_i$ and $t_{i+1}$, we obtain
		\begin{align}
			\Delta Y_i =  Y_{t_{i+1}}-Y_{t_i}=-\integral \psi(\Theta_s)ds +\integral\prom{Z_s,\cdot}_0 dW_s.
			\label{eq:resta}
		\end{align}
		Using the definition of $\widehat{\ca{V}}_{t_i}$ in \ref{eq:nu-gorro},
		\[
		\begin{aligned}
			Y_{t_i}-\widehat{\ca{V}}_{t_i}=&~{}  Y_{t_{i+1}} -\Delta Y_i -\widehat{\ca{V}}_{t_i}\\
			=&~{} Y_{t_{i+1}}+\integral [\psi(\Theta_s)-\psi(\widehat{\Theta}_{t_i})]ds-\integral \prom{Z_s,\cdot}_0 dW_s -\bb{E}_i\ug{i+1}(X^\pi_{t_{i+1}}).
		\end{aligned}
		\]
		Here $\widehat{\Theta}_{t_i}=(t_i,X^\pi_{t_i},\widehat{\ca{V}}_{t_i}, \gb{Z}_{t_i})$. Then, by taking $\bb{E}_i$ and using that stochastic integration produces a martingale
		\begin{align*}
			Y_{t_i}-\widehat{\ca{V}}_{t_i}=\bb{E}_i(Y_{t_{i+1}}-\ug{i+1}(X^\pi_{t_{i+1}})) + \bb{E}_i\parent{\integral [\psi(\Theta_s)-\psi(\widehat{\Theta}_{t_i})]ds} = a+b.
		\end{align*}
		Using the classical inequality $(a+b)^2\le (1+\gamma h)a^2+(1+\frac{1}{\gamma h})b^2$ for $\gamma>0$ to be chosen, we get
		\begin{equation}\label{parada1}
			\begin{aligned}
				\bb{E}\barras{Y_{t_i}-\widehat{\ca{V}}_{t_i}}^2\le &~{} (1+\gamma h)
				\bb{E} \left[ \bb{E}_i\parent{Y_{t_{i+1}}-\ug{i+1}(X^\pi_{t_{i+1}})} \right] ^2  \\
				&~{} + \parent{1+\frac{1}{\gamma h}}\bb{E} \left[\bb{E}_i\parent{\integral [\psi(\Theta_s)-\psi(\widehat{\Theta}_{t_i})]ds}\right]^2.
			\end{aligned}
		\end{equation}
		With no lose of generality, as we are seeking for an upper bound, we can replace $[\psi(\Theta_s)-\psi(\widehat{\Theta}_{t_i})]$ by $|\psi(\Theta_s)-\psi(\widehat{\Theta}_{t_i})|$. Also, in the second term, we can drop the $\bb{E}_i$ due to the law of total expectation. The Lipschitz condition on $\psi$ in Assumptions \ref{assumptions} allows us to give an upper bound in terms of the difference between $\Theta_s$ and $\widehat{\Theta}_{t_i}$. Indeed, we have that
		\begin{align*}
			\bb{E} \left[\bb{E}_i\parent{\integral [\psi(\Theta_s)-\psi(\widehat{\Theta}_{t_i})]ds}\right]^2 \le &~{} Ch\left[ h^2+\bb{E}\integral \norm{X_s-X^{\pi}_{t_i}}_H^2ds + \bb{E}\integral |Y_s-\vg{t_i}|^2ds \right.\\
			&\qquad +\left. \bb{E}
			\integral \norm{Z_s-\gb{Z}_{t_i}}_V^2ds \right],
		\end{align*}
		where the Lipschitz constant of $\psi$ was absorbed by $C$. Using now triangle inequality $|Y_s-\vg{t_i}| \leq |Y_s-Y_{t_i}| +|Y_{t_i}-\vg{t_i}|$  and the definition of $e_i$ in (\ref{eq:e-def}), we find
		\begin{align}
			 \bb{E} \left[\bb{E}_i\parent{\integral [\psi(\Theta_s)-\psi(\widehat{\Theta}_{t_i})]ds}\right]^2\le\ & C h\left[  h^2+ e_i(X,\xscheme{i}) + 
			e_i(Y,Y_{t_i}) +h \esp{Y_{t_i}- \vg{t_i}}^2 \right. \nonumber\\
			&\left. \qquad\qquad
			+ \bb{E}\integral \norm{Z_s-\gb{Z}_{t_i}}_V^2 ds \right]. 
		\end{align}
		For the sake of brevity, define now 
		\begin{equation}\label{def_H}
			H_{i}:=Y_{t_{i}}-\ug{i} (X^\pi_{t_{i}}).
		\end{equation}
		Therefore, replacing in \eqref{parada1},
		\begin{align}
			& \esp{Y_{t_i}-\widehat{\ca{V}}_{t_i}}^2 \le \left(1+\gamma h \right)\bb{E}|\bb{E}_i H_{i+1}|^2 + \parent{1+\gamma h} \frac{C}{\gamma} \left[ h^2 + e_i(X,\xscheme{i}) + 
			e_i(Y,Y_{t_i}) +h \esp{Y_{t_i}- \vg{t_i}}^2 \right. \nonumber\\
			&\left. \qquad \qquad  \qquad  \quad \quad
			+ \bb{E}\integral \norm{Z_s-\gb{Z}_{t_i}}_V^2 ds \right].
			\label{eq:upper} 
		\end{align}
		Recall $\overline Z_{t_i}$ introduced in equation (\ref{eq:z-barra}). In order to work with last term in previous equation, we prove the following,
		\begin{align}
			\bb{E}\integral \norm{Z_s-\gb{Z}_{t_i}}_V^2 ds&=\bb{E}\integral\norm{Z_s-\overline{Z}_{t_i}}_V^2 ds+h \bb{E}\norm{\overline{Z}_{t_i}-\gb{Z}_{t_i}}_V^2.
			\label{eq:ortogonal1}
		\end{align}
		Indeed,
		\begin{align*}
			\norm{Z_t-\gb{Z}_{t_i}}_V^2 =&~{}  \norm{ (Z_t-\overline{Z}_{t_i}) + (\overline{Z}_{t_i}-\gb{Z}_{t_i})}_V^2 \\
			=&~{} \norm{Z_t-\overline{Z}_{t_i}}_V^2 + \norm{\overline{Z}_{t_i}-\gb{Z}_{t_i}}_V^2 + 2\prom{Z_t-\overline{Z}_{t_i}, \overline{Z}_{t_i}-\gb{Z}_{t_i}}_V.
		\end{align*}

		It is sufficient to establish that the double product is null when we integrate and take expected valued. Recall that $\overline{Z}_{t_i}$ from (\ref{eq:z-barra})
		is a $\ca{F}_{t_i}$ measurable random variable. Then, by using elementary properties of Bochner integral,  
		\begin{align*}
			\bb{E}\integral \prom{Z_t-\overline{Z}_{t_i}, \overline{Z}_{t_i}-\gb{Z}_{t_i}}_Vdt& = \bb{E}\Big\langle\integral (Z_s-\overline{Z}_{t_i})ds ,\overline{Z}_{t_i}-\gb{Z}_{t_i} \Big\rangle_V \\
			& = h\bb{E}\Big\langle\frac{1}{h}\integral Z_s ds-\overline{Z}_{t_i} ,\overline{Z}_{t_i}-\gb{Z}_{t_i}\Big\rangle_V = 0.
		\end{align*}
		The latter is due to the fact that $\overline{Z}_{t_i}-\gb{Z}_{t_i}\in L^2(\Omega,\ca{F}_{t_i},\p;V)$ and $\frac{1}{h}\integral Z_s ds-\overline{Z}_{t_i}$ is an orthogonal element to $L^2(\Omega,\ca{F}_{t_i},\p;V)\subset L^2(\Omega,\ca{F},\p;V)$. Therefore, equation (\ref{eq:ortogonal1}) is established. By multiplying (\ref{eq:resta}) by $\Delta W_i$ and taking $\bb{E}_i$,
		\begin{align*}
			\bb{E}_i\parent{\Delta W_iY_{t_{i+1}} } + \bb{E}_i\parent{\Delta W_i\integral \psi(\Theta_s)ds } =&~{}  \bb{E}_i\parent{\integral dW_s\integral \prom{Z_s,\cdot}_0 dW_s }\\
			=&~{} \bb{E}_i\integral Z_s ds = h \overline{Z}_{t_i},
		\end{align*}
		where we used the arguments from the proof of Lemma \ref{lemma:z-gorro}. Subtracting $h \gb{Z}_{t_i}=\bb{E}_i(\ug{i+1}(X^{\pi}_{t_{i+1}})\Delta W_i)$ and then noting that $\bb{E}_i(\Delta W_i \bb{E}_i(H_{i+1}))=0$,
		\begin{align*}
			h (\overline{Z}_{t_i}- \gb{Z}_{t_i})= & \bb{E}_i\left[\Delta W_i (Y_{t_{i+1}}-\ug{i+1}(X^{\pi}_{t_{i+1}}))\right] + \bb{E}_i\parent{\Delta W_i\integral \psi(\Theta_s)ds}\\
			=& \bb{E}_i\left[\Delta W_i (H_{i+1} - \bb{E}_i H_{i+1})\right] + \bb{E}_i\parent{\Delta W_i\integral \psi(\Theta_s)ds}
		\end{align*}
		By applying the conditional version of Holder inequality for the first term and its classical form to the second one, follows that
		\begin{align}
			h^2\bb{E}\norm{\overline{Z}_{t_i} -\gbt{Z}{i} }_V^2 &= \bb{E}\norm{\bb{E}_i\Big[ \Delta W_i (H_{i+1} - \bb{E}_i H_{i+1}) \Big] + \bb{E}_i\parent{\Delta W_i \integral \psi(\Theta_s)ds} }_V^2\nonumber\\
			&\le 2\bb{E}\parent{\bb{E}_i\norm{\Delta W_i}^2_V\bb{E}_i[H_{i+1}-\bb{E}H_{i+1}]^2} + 2\bb{E}\parent{\bb{E}_i\norm{\Delta W_i}_V^2\bb{E}_i\Bigg[\integral \psi(\Theta_s)ds\Bigg]^2}\nonumber\\
			&\le  C\text{tr}(Q)\bb{E}\parent{\bb{E}_i H_{i+1}^2 - (\bb{E}_i H_{i+1})^2} + Ch\, \text{tr}(Q)\bb{E}\integral |\psi(\Theta_s)|^2ds;
			\label{eq: 4.11z}
		\end{align}
		Putting all together, 
		\begin{align*}
			\esp{Y_{t_i}-\widehat{\ca{V}}_{t_i}}^2
			&\le \left(1+\gamma h\right)\esp{\bb{E}_i(H_{i+1}) }^2 \\
			&\quad + \parent{1+\gamma h} \frac{C}{\gamma} \Big[  h^2+ e_i(X,\xscheme{i})+e_i(Y,Y_{t_i})+e_i(Z,\overline{Z}_{t_i})+h \bb{E}|Y_{t_i}- \vg{t_i}|^2\\
			&\qquad \qquad\qquad \quad+\text{tr}(Q)\bb{E} H_{i+1}^2 - \text{tr}(Q)\bb{E}|\bb{E}_i H_{i+1}|^2\\
			&\qquad \qquad\qquad \quad + h\text{tr}(Q)\bb{E}\integral |\psi(\Theta_s)|^2ds\Big]
		\end{align*}
		Where we also used that $Z_t,\overline{Z}_{t_i}$ are $V_0$-valued and implies $\norm{Z_t-\overline{Z}_{t_i}}_V^2\le\norm{Q^{1/2}}_{L(Q)}^2\norm{Z_t-\overline{Z}_{t_i}}_0^2$. Let $\gamma=C^2\text{tr}(Q)$ and note that $(1+\gamma h)\frac{C}{\gamma}\le C$ and also $\gamma\le C$, then the above term transform to
		\begin{align*}
			& Ch^2 + Ce_i(X,\xscheme{i}) + Ce_i(Y,Y_{t_i}) + Ce_i(Z,\overline{Z}_{t_i}) \\
			& + Ch\bb{E}|Y_{t_i}-\vg{t_i}|^2 + C(1+Ch)\bb{E}H_{i+1}^2 + Ch\bb{E}\integral |\psi(\Theta_s)|^2 ds.
		\end{align*}
		Now we take $h$ small such that $Ch < 1$ and then
		\begin{align*}
			\esp{Y_{t_i}-\widehat{\ca{V}}_{t_i}}^2 \le &~{} Ch^2 + Ce_i(X,\xscheme{i}) + Ce_i(Y,Y_{t_i}) + Ce_i(Z,\overline{Z}_{t_i}) \\
			&~{} + C(1+Ch)\bb{E}H_{i+1}^2 + Ch\bb{E}\integral |\psi(\Theta_s)|^2 ds.
		\end{align*}
		Finally, by recalling that $H_{i+1} =Y_{t_{i+1}}-\ug{i+1} (X^\pi_{t_{i+1}})$, we have established \eqref{step1}.
	\end{proof}
	
	{\color{black} \textbf{Step 2:}} The term,
	\[
	C(1+Ch)\bb{E} \left| Y_{t_{i+1}}-\ug{i+1} (X^\pi_{t_{i+1}}) \right|^2,
	\]
	 in \eqref{step1} was left without a control in previous step. Here in what follows we provide a control on this term. The purpose of this section is to show the following estimate:
	\begin{lemma}\label{lemma:first-bound}
		There exists a constant $C>0$ such that,
		\begin{align}
			\max_{i\in\set{0,...,N-1}}\esp{Y_{t_i}-\ug{i}(X_{t_i}^\pi)}^2&\le C\Bigg[h+\esp{ \phi(X_T)-\phi(X_T^\pi) }^2N+ \sum_{i=0}^{N-1} \esp{\ug{i}(X_{t_i}^\pi)-\widehat{\ca{V}}_{t_i}}^2 \nonumber\\
			&\qquad ~{} + e(X,X^{\pi}) + e(Y,(Y_t)_{t\in\pi}) + e(Z,(\overline{Z}_t)_{t\in\pi})  \Bigg]. \label{step2}
		\end{align}
	\end{lemma}
	The rest of this section is devoted to the proof of this result.
	
	\begin{proof}[Proof of Lemma \ref{lemma:first-bound}]
		Recall $H_{i+1} =Y_{t_{i+1}}-\ug{i+1} (X^\pi_{t_{i+1}})$. We have that $(a+b)^2\ge(1-h)a^2+(1-\frac{1}{h})b^2$ and
		\begin{align}
			\esp{Y_{t_i}-\widehat{\ca{V}}_{t_i}}^2&=\esp{\parent{Y_{t_i}-\ug{i}(X^\pi_{t_i})} + \parent{\ug{i}(X^\pi_{t_i})-\widehat{\ca{V}}_{t_i}}}^2
			\label{eq:4.13}\\
			&\ge (1-h)\esp{Y_{t_i}-\ug{i}(X^\pi_{t_i})}^2 + \parent{1-\frac{1}{h}}\esp{\ug{i}(X^\pi_{t_i})-\widehat{\ca{V}}_{t_i}}^2.\nonumber
		\end{align}
		Therefore, we have an upper \eqref{step1} and lower \eqref{eq:4.13} bound for $\esp{Y_{t_i}-\widehat{\ca{V}}_{t_i}}^2$. By connecting these bounds, 
		\begin{align*}
			(1-h)\esp{Y_{t_i}-\ug{i}(X^\pi_{t_i})}^2 + \parent{1-\frac{1}{h}}\esp{\ug{i}(X^\pi_{t_i})-\widehat{\ca{V}}_{t_i}}^2 &\ \le Ch^2+Ce_i(X,\xscheme{i}) + Ce_i(Y,Y_{t_i}) + Ce_i(Z,\overline{Z}_{t_i})\\
			&\quad \quad +Ch \bb{E}\integral \psi(\Theta_s)^2 ds+ C(1+Ch)\bb{E}\parent{H_{i+1}^2}.
		\end{align*}
		Using that for sufficiently small $h$ we have $(1-h)^{-1}\le 2\le C$, we get,
		\begin{align*}
			\esp{Y_{t_i}-\ug{i} (\xscheme{i})}^2 &\le CN\esp{\ug{i}(\xscheme{i})-\widehat{\ca{V}}_{t_i}}^2 + Ch^2 + C e_i(X,\xscheme{i}) + C e_i(Y,Y_{t_i})+C e_i(Z,\overline{Z}_{t_i})  \\
			& ~{}+   Ch\bb{E}\integral|\psi(\Theta_s)|^2ds+C\esp{Y_{t_{i+1}}-\ug{i+1}(X^{\pi}_{t_{i+1}}) }^2.
		\end{align*}
		Notice that the expression on time $t_i$ that we want to estimate, appears on the right side on time $t_{i+1}$, we can iterate the bound and get that $\forall$ $i\in\set{0,...,N-1}$
		\begin{align*}
			& \esp{Y_{t_i}-\ug{i} (\xscheme{i})}^2 \\
			&~{} \le CN\sum_{k=i}^{N-1}\esp{\ug{k}  (\xscheme{k}) -\widehat{\ca{V}}_{t_k}}^2 + C(N-i)h^2 + C\sum_{k=i}^{N-1}\left[ e_i(X,\xscheme{i}) + e_i(Y,Y_{t_i}) + e_i(Z,\overline{Z}_{t_i})\right] \\
			&\quad+Ch\sum_{k=i}^{N-1}\bb{E}\int_{t_k}^{t_{k+1}} |\psi(\Theta_s)|^2ds +C\esp{Y_{t_{N}}-\phi(\xscheme{N})}^2\\
			&~{} \leq CN\sum_{k=0}^{N-1}\esp{\ug{k} (X_{t_k}^\pi) -\widehat{\ca{V}}_{t_k}}^2 + CNh^2 + C\left[ e(X,X^{\pi}) + e(Y,(Y_t)_{t\in\pi}) + e(Z,(\overline{Z}_t)_{t\in\pi}) \right] \\
			&~{} \quad + Ch\sum_{k=0}^{N-1}\bb{E}\int_{t_k}^{t_{k+1}}| \psi(\Theta_s)|^2ds +C\esp{Y_{t_{N}}-\phi(\xscheme{N})}^2.
		\end{align*}
		Applying maximum on $i\in\set{0,...,N-1}$ and recalling bound from Lemma (\ref{lemma:f-int-bound}),
		\begin{align}
			\max_{i\in\set{0,...,N-1}}\esp{Y_{t_i}-\ug{i}(X_{t_i}^\pi)}^2&\le C\Bigg[h+\esp{ \phi(X_T)-\phi(X_T^\pi) }^2N \sum_{i=0}^{N-1} \esp{\ug{i}(X_{t_i}^\pi)-\widehat{\ca{V}}_{t_i}}^2 \nonumber\\
			&\qquad + e(X,X^{\pi}) + e(Y,(Y_t)_{t\in\pi}) + e(Z,(\overline{Z}_t)_{t\in\pi})  \Bigg]. \label{step2}
		\end{align}
		This is nothing that \eqref{step2}.
	\end{proof}

	{\color{black} \textbf{Step 3:}} Estimate \eqref{step2} contains some uncontrolled terms on its RHS. Here the purpose is to bound the term
	\[
	\sum_{i=0}^{N-1} \esp{\ug{i}(X_{t_i}^\pi)-\widehat{\ca{V}}_{t_i}}^2,
	\]
	in terms of more tractable terms. In this step we will prove
	\begin{lemma}\label{lemma:important-lemma}
		It holds that,
		\begin{align}
			\esp{\hat{u}_i(\xscheme{i})-\widehat{\ca{V}}_{t_i}}^2 +h \bb{E}\norm{\gb{Z}_{t_i}-\hat{z}_i (\xscheme{i})}_0^2\le C\varepsilon^v_i + Ch\varepsilon^z_i,
			\label{eq:important-bound}
		\end{align}
		with $\varepsilon_i^v$ and $\varepsilon_i^z$ defined in \eqref{errores}.
	\end{lemma}
	
	\begin{proof}
		Fix $i\in\set{0,...,N-1}$. Recall the martingale $(N_t)_{t\in[t_i,t_{i+1}]}$ and take $t=t_{i+1}$,
		\begin{align*}
			\ug{i+1} (X^{\pi}_{t_{i+1}}) = \bb{E}_i \ug{i+1}(X^{\pi}_{t_{i+1}}) + \int_{t_i}^{t_{i+1}} \prom{\widehat{Z}_s,\cdot}_0 dW_s.
		\end{align*}
		Now we replace the definition of $\widehat{\ca{V}}_{t_i}$ (\ref{eq:nu-gorro}),
		\begin{align}
			\ug{i+1} (X^{\pi}_{t_{i+1}}) = \widehat{\ca{V}}_{t_i} - \psi(t_i,X^{\pi}_{t_i},\widehat{\ca{V}}_{t_i}, \gb{Z}_{t_i})h +\int_{t_i}^{t_{i+1}} \prom{\widehat{Z}_s,\cdot}_0 dW_s.
			\label{eq:rep}
		\end{align}
		Now fix a parameter $\theta\in\Theta_{\eta}$ and replace (\ref{eq:rep}) on $L_i(\theta)$:
		\begin{align*}
			L_i(\theta) =&~{} \bb{E}\Big| \widehat{\ca{V}}_{t_i} - u^{\theta}_i(\xscheme{i})
			+\psi(t_i,\xscheme{i},u^{\theta}_i(\xscheme{i}),z^{\theta}_i(\xscheme{i}))h-\psi(t_i,\xscheme{i},\widehat{\ca{V}}_{t_i},\gb{Z}_{t_i})h+ \integral \prom{\widehat{Z}_s-z^{\theta}_i(\xscheme{i}),\cdot}_0 dW_s\Big|^2 
		\end{align*}
		Note that the four first terms are $\ca{F}_{t_i}$-measurable and the  stochastic integral is a martingale difference, therefore
		\begin{align*}
			L_i(\theta) =&~{} \bb{E}\Big| \widehat{\ca{V}}_{t_i} - u^{\theta}_i(\xscheme{i})
			+\psi(t_i,\xscheme{i},u^{\theta}_i(\xscheme{i}),z^{\theta}_i(\xscheme{i}))h-\psi(t_i,\xscheme{i},\widehat{\ca{V}}_{t_i},\gb{Z}_{t_i})h\Big|^2\\
			&\quad 
			+ \bb{E}\integral \norm{\widehat{Z}_s-\gb{Z}_{t_i}}^2_0 ds + h\bb{E}\norm{\gb{Z}_{t_i}-z^{\theta}_i(\xscheme{i})}^2_0.
		\end{align*}
		Where we used Ito isometry and the same argument used on equation (\ref{eq:ortogonal1}). With this decomposition of $L_i(\theta)$, we can easily see the part that depends on $\theta$. Lets work with $\hat{L}_i$ defined as follows,
		\begin{align*}
			\hat{L}_i(\theta) &= \bb{E} \Big| \widehat{\ca{V}}_{t_i}-u_i^{\theta}(\xscheme{i}) + \parent{ \psi(t_i,\xscheme{i},u_i^{\theta}(\xscheme{i}),z_i^{\theta}(\xscheme{i})) - \psi(t_i,\xscheme{i},\widehat{\ca{V}}_{t_i},\gb{Z}_{t_i})}h \Big|^2+ h \bb{E}\norm{\gb{Z}_{t_i}-z_i^{\theta}(\xscheme{i})}_0^2.
		\end{align*}
		Let $\gamma>0$ and use Young inequality and the Lipschitz condition on $\psi$ to find that
		\begin{align*}
			&\bb{E} \Big| \widehat{\ca{V}}_{t_i}-u_i^{\theta}(\xscheme{i}) + \parent{\psi(t_i,\xscheme{i},\widehat{\ca{V}}_{t_i},\gb{Z}_{t_i}) - \psi(t_i,\xscheme{i},u_i^{\theta}(\xscheme{i}),z_i^{\theta}(\xscheme{i})) } \Big|^2\\
			&\le \parent{1+\gamma h}\esp{\widehat{\ca{V}}_{t_i}-u_i^{\theta}(\xscheme{i})}^2+  \parent{1+\frac{1}{\gamma h}}h^2 C\bb{E}\parent{|\widehat{\ca{V}}_{t_i}-u_i^{\theta}(\xscheme{i})|^2+\norm{z_i^{\theta} (\xscheme{i})-\gb{Z}_{t_i}}_0^2}\\
			&\le C\esp{\widehat{\ca{V}}_{t_i}-u_i^{\theta}(\xscheme{i})}^2 + Ch\bb{E}\norm{z_i^{\theta} (\xscheme{i})-\gb{Z}_{t_i}}_0^2.
		\end{align*}
		Therefore, we have an upper bound on $L(\theta)$ for all $\theta\in\Theta_{\eta}$, to find a lower bound, we use $(a+b)^2\ge(1-\gamma h)a^2+\left(1-\frac{1}{\gamma h}\right)b^2\ge(1-\gamma h)a^2-\frac{1}{\gamma h}b^2$ with $\gamma>0$
		\begin{align*}
			\bb{E} \Big| \widehat{\ca{V}}_{t_i}-u_i^{\theta}(\xscheme{i}) + \parent{\psi(t_i,\xscheme{i},\widehat{\ca{V}}_{t_i},\gb{Z}_{t_i}) - \psi(t_i,\xscheme{i},u_i^{\theta}(\xscheme{i}),z_i^{\theta}(\xscheme{i})) } \Big|^2\ge& \parent{1-Ch}\esp{\widehat{\ca{V}}_{t_i}-u_i^{\theta}(\xscheme{i})}^2 \\
			&-\frac{h}{2} \bb{E}\norm{z_i^{\theta} (\xscheme{i})-\gb{Z}_{t_i}}_0^2;
		\end{align*}
		where we used $\gamma = 2C$ in order to force the $\frac{1}{2}$ in the second term of the RHS. Then, connecting these bounds and using that $\forall\theta\in\Theta$ $\hat{L}(\theta^*)\le\hat{L}(\theta)$ yields,
		\begin{align*}
			\parent{1- Ch}\esp{\widehat{\ca{V}}_{t_i}-\hat{u}_i(\xscheme{i})}^2 +\frac{h}{2} \bb{E}\norm{\gb{Z}_{t_i}-\hat{z}_i (\xscheme{i})}_0^2\le C\esp{\widehat{\ca{V}}_{t_i}-u_i^{\theta}(\xscheme{i})}^2 Ch\bb{E}\norm{\gb{Z}_{t_i}-z_i^{\theta} (\xscheme{i})}_0^2.
		\end{align*}
		By taking $h$ small such that $(1-Ch)\ge \frac{1}{2}$ and infimum on the right side with respect to $\theta\in\Theta_{\eta}$ we get \eqref{eq:important-bound},
		\begin{align}
			\esp{\widehat{\ca{V}}_{t_i}-\hat{u}_i(\xscheme{i})}^2 +h \bb{E}\norm{\gb{Z}_{t_i}-\hat{z}_i (\xscheme{i})}_0^2\le C\varepsilon^{v,\eta}_i + Ch\varepsilon^{z,\eta}_i
		\end{align}
		Thus the proof is completed.
	\end{proof}
Previous lemma and steps proves the following.
\begin{lemma}\label{lemma:first-control}
	It holds that,
	\begin{align}
		\max_{i\in\set{0,...,N-1}}\esp{Y_{t_i}-\ug{i}(X_{t_i}^\pi)}^2 + \le& C\Bigg[h+\esp{ \phi(X_T)-\phi(X_T^\pi) }^2+N\varepsilon^{v,\eta} + \varepsilon^{z,\eta}\nonumber\\
		&\quad + e(X,X^{\pi}) + e(Y,(Y_t)_{t\in\pi})
		+ e(Z,(\overline{Z}_t)_{t\in\pi}) \Bigg].
		\label{eq:first-control}
	\end{align}
\end{lemma}
	
	{\color{black} \textbf{Step 4:}} In this step we show the desire bound for the remaining component. 
	\begin{lemma}\label{lemma:second-control}
		It holds that,
		\begin{align}
			\sum_{i=0}^{N-1}\bb{E}\integral \norm{Z_s-\hat{z}_{i}(\xscheme{i})}_0^2  ds\le& C\,\Bigg[h +\esp{ \phi(X_T)-\phi(X_T^\pi) }^2+ N\varepsilon^{v,\eta} + \varepsilon^{z,\eta}\nonumber\\
			&\quad +e(X,X^{\pi}) + e(Y,(Y_t)_{t\in\pi})+ e(Z,(\overline{Z}_t)_{t\in\pi}) \Bigg].
			\label{eq:second-control}
		\end{align}
	\end{lemma}
\begin{proof}
	We will use triangular inequality passing through $\gb{Z}_{t_i}$. Note that the term containing $\norm{\gb{Z}_{t_i} - \hat{z}_i(\xscheme{i})}^2_0$ is well-controlled by Lemma \ref{lemma:important-lemma}. By using (\ref{eq: 4.11z}) with Lemma \ref{lemma:f-int-bound} on (\ref{eq:ortogonal1}), we get
	\begin{align*}
		\bb{E}\integral\norm{Z_s - \gb{Z}_{t_i}}_0^2 ds \le &~{} C\bb{E}\integral \norm{Z_s - \overline{Z}_{t_i}}_0^2 ds +C \bb{E}\parent{\bb{E}_i H_{i+1}^2 - (\bb{E}_i H_{i+1})^2}\\
		&~{} + Ch\bb{E}\integral |\psi(\Theta_s)|^2ds .
	\end{align*}
	which implies, after summing over $i\in\set{0,...N-1}$,
	\begin{align}\label{eq:final}
		\bb{E}\sum_{i=0}^{N-1}\integral \norm{Z_t - \gb{Z}_{t_i}}_0^2 ds\le  C\sum_{i=0}^{N-1}\parent{\bb{E}\parent{H_{i+1}^2} - \esp{\bb{E}_i (H_{i+1})}^2} + Ch + e(Z,(\overline{Z}_t)_{t\in\pi}).
	\end{align}
	The next step is to give a suitable bound for $\bb{E}\parent{H_{i+1}^2} - \esp{\bb{E}_i (H_{i+1})}^2$. Recall from \eqref{def_H} that $H_{i+1} = Y_{t_{i+1}} - \hat{u}_{i+1} (\xscheme{i+1})$, then
	\begin{equation}\label{EHEH}
		\begin{aligned}
			\sum_{i=0}^{N-1}\parent{\bb{E}\parent{H_{i+1}^2} - \esp{\bb{E}_i (H_{i+1})}^2} &= \sum_{i=0}^{N-1}\bb{E}(H_{i+1}^2)-\sum_{i=0}^{N-1}\esp{\bb{E}_i(H_{i+1})}^2\\
			&= \esp{Y_{t_N}-\hat{u}_N(\xscheme{N})} + \sum_{i=0}^{N-2}\bb{E}(H_{i+1}^2) - \sum_{i=0}^{N-1}\esp{\bb{E}_i(H_{i+1})}^2\\
			&\le \esp{\phi(X_T)-\phi(X_T^\pi)}^2 +\bb{E}(H_{0}^2) + \sum_{i=1}^{N-1}\bb{E}(H_{i}^2) - \sum_{i=0}^{N-1}\esp{\bb{E}_i(H_{i+1})}^2\\
			&=  \esp{\phi(X_T)-\phi(X_T^\pi)}^2  +\sum_{i=0}^{N-1} \left( \bb{E}(H_{i}^2) - \esp{\bb{E}_i(H_{i+1})}^2 \right).
		\end{aligned}
	\end{equation}
	From (\ref{eq:4.13}) and (\ref{eq:upper}) we have an lower and upper bound for $\esp{Y_{t_i}-\widehat{\ca{V}}_{t_i}}^2$. Indeed, first one has
	\begin{equation}\label{eq:delta-s4}
		(1-h)\esp{Y_{t_i}-\ug{i}(X^\pi_{t_i})}^2 \leq   \esp{Y_{t_i}-\widehat{\ca{V}}_{t_i}}^2   + \parent{\frac{1}{h}-1}\esp{\ug{i}(X^\pi_{t_i})-\widehat{\ca{V}}_{t_i}}^2.
	\end{equation}
	Then, we have that for all $\gamma>0$
	\begin{align*}
		&\parent{1-h}\, \esp{Y_{t_i}-\hat{u}_i(\xscheme{i})}^2  \\
		&~{}  \le \parent{\frac{1}{h} -1}\esp{\hat{u}_i(\xscheme{i})-\widehat{\ca{V}}_{t_i}}^2 + (1+\gamma h)\esp{\bb{E}_i(H_{i+1})}^2 
		\\
		& \quad +(1+\gamma h)\frac{C}{\gamma}\bigg[\underbrace{h^2+ e_i(X,\xscheme{i}) + e_i(Y,Y_{t_i}) + h\bb{E}|Y_{t_i} - \widehat{\ca{V}}_{t_i}|^2 + \bb{E}\integral \norm{Z_s-\gb{Z}_{t_i}}_0^2 ds}_{B_i}\bigg].
	\end{align*}
	Let us call the expression inside the squared brackets by $B_i$. Subtracting $(1-h) \bb{E}\left|\bb{E}_iH_{i+1}\right|^2$ and dividing by $(1-h)$,
	\[
	\begin{aligned}
		\bb{E}(H_i^2) - \esp{\bb{E}_i(H_{i+1})}^2 \le &~{} \frac{1}{h}\esp{\ug{i}(\xscheme{i}) - \widehat{\ca{V}}_{t_i} }^2 + \parent{\frac{h+ \gamma h}{1-h}}\esp{\bb{E}_i(H_{i+1})}^2 +\frac{C}{\gamma}\frac{(1+ \gamma h)}{(1-h)}B_i.
	\end{aligned}
	\]
	For $\gamma = 3C$ and sufficiently small $h$, we can force,
	\begin{align*}
		\frac{C}{\gamma}\frac{(1+\gamma h)}{(1-h)} \le \frac{1}{2}\qquad \text{and}\qquad \frac{1}{1-h}\le\frac{1}{2}.
	\end{align*}
	Hence,
	\begin{align*}
		\bb{E}(H_i^2)& - \esp{\bb{E}_i(H_{i+1})}^2 \le \frac{1}{h}\esp{\ug{i}(\xscheme{i}) - \widehat{\ca{V}}_{t_i} }^2 + Ch\esp{\bb{E}_i(H_{i+1})}^2
		+\frac{1}{2}B_i.
	\end{align*}
	Finally, note that,
	\begin{align}\label{eq:N-bound}
		\sum_{i=0}^{N-1}\esp{\bb{E}_i(H_{i+1})}^2\le \esp{\phi(X_T)-\phi(X_T^\pi)}^2+N\underset{i=0,...,N-1}{\max} \esp{Y_{t_i}-\hat{u}_i(\xscheme{i})}^2.
	\end{align}
	Coming back to \eqref{EHEH},
	\[
	\begin{aligned}
		\sum_{i=0}^{N-1}\parent{\bb{E}\parent{H_{i+1}^2} - \esp{\bb{E}_i (H_{i+1})}^2}  \leq &~{}  C\esp{\phi(X_T)-\phi(X_T^\pi)}^2  + N \sum_{i=0}^{N-1}\esp{\ug{i}(\xscheme{i}) - \widehat{\ca{V}}_{t_i} }^2  \\
		&~{} + Ch N \underset{i=0,...,N-1}{\max} \esp{Y_{t_i}-\hat{u}_i(\xscheme{i})}^2 +\frac{1}{2} \sum_{i=0}^{N-1} B_i.
	\end{aligned}
	\]
	Therefore, by plugging this bound in (\ref{eq:final}), noting that $|Y_{t_i}-\widehat{\ca{V}}_{t_i}|^2\le 2|Y_{t_i}-\hat{u}_{i}(\xscheme{i}) |^2 + 2|\hat{u}_i(\xscheme{i}) - \widehat{\ca{V}}_{t_i}|^2$ and $hN = 1$, we have,
	\begin{align*}
		& \bb{E}\sum_{i=0}^{N-1}\integral \norm{Z_s - \gb{Z}_{t_i}}_0^2 ds\\
		 &\le  C\bigg[h + \esp{\phi(X_T)-\phi(X_T^\pi)}^2 + N \sum_{i=0}^{N-1}\esp{\hat{u}_{t_i}(\xscheme{i}) - \widehat{\ca{V}}_{t_i}}^2\\
		&\qquad + \underset{i=0,...,N-1}{\max}\esp{Y_{t_i}-\hat{u}_{i}(\xscheme{i})}^2 + e(X,X^{\pi}) + e(Y,(Y_t)_{t\in\pi}) + e(Z,(\overline{Z}_t)_{t\in\pi})\bigg].
	\end{align*}
	Now, use Lemma \ref{lemma:important-lemma} and Lemma \ref{lemma:first-control} to get
	\begin{align*}
		\bb{E}\sum_{i=0}^{N-1}\integral \norm{Z_s - \gb{Z}_{t_i}}_0^2 ds \le & C\bigg[h+\esp{\phi(X_T)-\phi(X_T^\pi)}^2+ N\varepsilon^{v,\eta} + \varepsilon^{z,\eta}\\
		&\quad +e(X,X^{\pi}) + e(Y,(Y_t)_{t\in\pi}) + e(Z,(\overline{Z}_t)_{t\in\pi})\bigg] .
	\end{align*}
	Thus, it has been demonstrated.
\end{proof}
By combining Lemma \ref{lemma:first-control} with Lemma \ref{lemma:second-control} and using Assumptions \ref{assumption:stochastic-sol-regularity}, the proof of Theorem \ref{MT1} is now complete.
\end{proof}

We finish this work with the following closing remark.

\begin{remark}
	Note that if the approximators are DeepOnets, then $\varepsilon^{v,\eta}, \varepsilon^{z,\eta}\to 0$ as $\eta\to\infty$. See Remark \ref{rem_5p4}.
\end{remark}

\subsection{Comments}

\begin{itemize}
	\item A natural generalization of the theory presented here is to consider an infinite dimension framework of the non-local case, i.e., extending our previous article \cite{yo} by replacing $\Rd$ with a separable Hilbert space $H$. The resulting PDE \eqref{eq:pde} would take the form
	\begin{equation}
		\left \{
		\begin{aligned}
			\partial_t u(t,x) + \ca{L}[u](t,x) + \psi \big( t,x,u(t,x),B^* (t,x)\nabla u(t,x),\ca{I}[u](t,x) \big)&=0, && (t,x)\in[0,T]\times H ,\\
			u(T,x) &= \phi(x), && x\in H.
		\end{aligned}
		\right.
		\label{eq:pide}
	\end{equation}
	Where $\ca{L}$ has an additional term of the form
	\begin{align*}
		\int_H \Big[ u(t,x+\beta(x,y)) - u(t,x) -\prom{\nabla u(t,x),\beta(x,y)}_H \Big]\lambda(dy),
	\end{align*}
	and the non-linearity $\psi$ has an extra, real and non-local dependence on the function $u$ given by
	\begin{align*}
		\ca{I}[u](t,x) = \int_{H} \parent{u(t,x+\beta(x,y)) - u(t,x)}\lambda(dy).
	\end{align*} 
	Where $\beta\colon H\times H\to H$ is a proper well-defined function. In the stochastic side, system \eqref{eq:fpsde}-\eqref{eq:bpsde} would have to consider an extra process  $\set{U(t,x)\colon (t,x)\in[0,T]\times H}$ and a Poisson random measure $\mu$ such that the stochastic equations takes the form
	\begin{align*}
		X_t &= x +\int_0^t \parent{AX_s + F(s,X_s)}ds + \int_0^t B(s,X_s)dW_s  + \int_0^t\int_H \beta(X_r,y)\bar{\mu}(dr,dy),\\
		Y_t &= \phi(X_T) + \int_t^T\psi(t,x,u(t,x),B^*(t,x)\nabla u(t,x),\ca{I}[u](t,x))ds - \int_t^T\prom{Z_s,\cdot}_0 dW_s - \int_t^T \int_H U(s,y)\bar{\mu}(ds,dy).
	\end{align*}
	See \cite{albeverio} for details on the theory of non-local SPDEs. It is left for future work to check whether the same techniques in \cite{yo} still holds for \eqref{eq:pide}.
	\item In Section \ref{sec:Universal Approximation Theorems and Deep-H-Onets} we aim to give an understanding of the success of deep learning via universal approximation theorems. Theorem \ref{theorem:infiniteapprox}, which stipulates the density of Deep-H-Onets in the space of squared integrable functions, is the main result of that section and we emphasize that it is key for the proof the first results concerning finite dimensional NNs given in the late eighties by Hornik et al. \cite{HSW89, HSW90, H91}.
	\item We would also like to state that our results does not give an answer to the curse of dimensionality (CoD) on this problem. Note that to avoid the CoD in this case can be defined by asking that the comlpexity of the problem does not grow exponentially on the inverse of the desired accuracy $\varepsilon$. Theorem \ref{MT1} provide bounds on the error committed by the proposed scheme, we do not show the explicit dependence on $\varepsilon$ for the constant $C$ that appears on \eqref{eq:MT-eq}.
\end{itemize}

\end{document}